\documentclass[12pt]{amsart}
\usepackage{graphicx}
\usepackage{amssymb}
\usepackage{amsmath}
\usepackage{amscd}
\usepackage{upgreek}
\usepackage{mathrsfs}
\usepackage{stmaryrd}
\usepackage{longtable}
\usepackage{txfonts}
\usepackage{extarrows}
\usepackage{titletoc}
\usepackage[T1]{fontenc}
\usepackage{eucal}
\usepackage{arydshln}
\usepackage{wrapfig}
\usepackage{float}
\usepackage{xypic}
\usepackage{dsfont}
\usepackage{xcolor}
\usepackage{color}
\usepackage[colorlinks,linkcolor=blue,anchorcolor=blue,
citecolor=blue]{hyperref}
\usepackage{hyperref}
\usepackage[hmarginratio=1:1,top=32mm,columnsep=20pt]{geometry}
\usepackage{amssymb,color}
\definecolor{MyGreen}{RGB}{29,162,55}
\DeclareMathOperator{\Hom}{Hom}	
\allowdisplaybreaks		
\usepackage{mathptmx}
\DeclareMathAlphabet{\mathcal}{OMS}{cmsy}{m}{n}
\DeclareSymbolFont{largesymbols}{OMX}{cmex}{m}{n}

\newtheorem{theorem}{Theorem}[section]
\newtheorem{prop}[theorem]{Proposition}
\newtheorem{defn}[theorem]{Definition}
\newtheorem{lem}[theorem]{Lemma}
\newtheorem{coro}[theorem]{Corollary}
\newtheorem{rem}[theorem]{Remark}

\newtheorem{exam}[theorem]{Example}
\newtheorem{thm}[theorem]{Theorem}

\newtheorem{notation}[theorem]{Notation}

\newcommand{\Z}{\mathbb{Z}} 
\newcommand{\F}{\mathbb{F}} 

\newcommand{\ra}{\longrightarrow}

\newcommand{\depth}{{\rm depth}}
\newcommand{\grade}{{\rm grade}}
\newcommand{\height}{{\rm height}}
\newcommand{\Tor}{{\rm Tor}}

\newcommand{\GL}{{\rm GL}} 
\newcommand{\Tr}{{\rm Tr}}

\newcommand{\HH}{\mathcal{H}}

\newcommand{\M}{\mathcal{M}}
\newcommand{\Sp}{{\rm Sp}}

\newcommand{\rightquot}{/\!\!/}

\newcommand{\maxpara}{\mathcal{G}}

\begin{document}
\setlength{\oddsidemargin}{0cm}
\setlength{\evensidemargin}{0cm}

\title{\scshape Modular invariants of finite gluing groups}

\author{\scshape Yin Chen}
\address{Department of Mathematics, Nantong University, Nantong 226019, P.R. China}
\email{ychen@ntu.edu.cn}

\author{\scshape R. James Shank}
\address{School of Mathematics, Statistics and Actuarial Science, University of Kent\\
\hfil\break\indent Canterbury, CT2 7FS, United Kingdom}
\email{R.J.Shank@kent.ac.uk}

\author{David L. Wehlau}
\address{Department of Mathematics and Computer Science, Royal Military College\\
\hfil\break\indent Kingston, ON, K7K 5L0, Canada}
\email{wehlau@rmc.ca}

\date{\today}
\def\shorttitle{Modular invariants of finite gluing groups}

\begin{abstract}
We use the gluing construction introduced by Jia Huang 
to explore the rings of invariants for a range of modular representations.
We construct generating sets for the rings of invariants of the maximal parabolic subgroups
of a finite symplectic group and their common Sylow $p$-subgroup. 
We also investigate the invariants of singular finite classical groups.
We introduce parabolic gluing and use this construction to compute the invariant field of fractions for a range of representations. 
We use thin gluing to construct faithful representations of semidirect products 
and to determine the minimum dimension of a faithful representation of the semidirect product of a cyclic $p$-group acting 
on an elementary abelian $p$-group. 
\end{abstract}


\maketitle
\baselineskip=16pt

\dottedcontents{section}[1.16cm]{}{1.8em}{5pt}
\dottedcontents{subsection}[2.00cm]{}{2.7em}{5pt}
\dottedcontents{subsubsection}[2.86cm]{}{3.4em}{5pt}


\section{Introduction}

In this paper we use the gluing construction introduced by Jia Huang \cite{Huang2011}
to explore the rings of invariants for a range of modular representations.
The gluing construction was motivated, in part, by the work of Hewett \cite{Hew1996},
Kuhn and Mitchell \cite{KM1986}, and Mui \cite{Mui1975} on parabolic subgroups of a finite general linear group.
In Section~\ref{max_sec}, we use gluing methods to construct generating sets for the rings of invariants of the maximal parabolic subgroups
of a finite symplectic group and their common Sylow $p$-subgroup. Our work in that section relies on the results of Carlisle and Kropholler on the invariants
of a finite symplectic group \cite[\S 8]{Ben1993}.
We also use the gluing construction to investigate the invariants of singular finite classical groups (Section~\ref{class_groups}). In Section~\ref{para_sec}, we introduce parabolic gluing and use this construction to compute the invariant field of fractions for a range of representations. We use thin gluing to construct faithful representations of semidirect products (Theorem \ref{thin_glue})
and to determine the minimum dimension of a faithful representation of the semidirect product of a cyclic $p$-group acting on an elementary abelian $p$-group (Corollary~\ref{min_dim}).

Suppose $V$ is a finite dimensional representation of a group $G$ over a field $\F$. We view $V$ as a left module over the group ring $\F G$.
There is a natural right action of $G$ on the dual $V^*=\hom_{\F}(V,\F)$: for $\phi\in V^*$, $g\in G$, and $v\in V$,
$(\phi\cdot g)(v)=\phi(g\cdot v)$. We use $\F[V]$ to denote the symmetric algebra on $V^*$.
The action of $G$ on $V^*$ extends to an action by degree preserving $\F$-algebra automorphisms on $\F[V]$.
The \emph{ring of invariants} of the representation is the subalgebra 
$\F[V]^G:=\{f\in\F[V]\mid f\cdot g=f , \, \forall g\in G\}$.
The elements of $\F[V]$ represent polynomial functions on $V$ and the elements of $\F[V]^G$ represent polynomial functions on the orbits $V/G$.
If $G$ is finite and $\F$ is algebraically closed, then $\F[V]^G$ is the ring of regular functions on the categorical quotient $V\rightquot G$.
For background material on the invariant theory of finite groups,
see  \cite{Ben1993}, \cite{CW2011}, \cite{DK2002},  \cite{FS2016} and \cite{NS2002}.
For background material on modular representation theory we suggest \cite{Alperin1986}.
We occasionally make reference to Steenrod operations; see \cite[\S 8]{NS2002} for the definition in the context of invariant theory.

In Section \ref{glue_constr} we introduce the gluing construction and define polynomial gluings, split gluings and thin gluings. 
In Section~\ref{tensor_sec} we summarise the relevant properties of tensor products of algebras. 
In Section~\ref{trans_sec} we compute the image of the transfer for a polynomial gluing in terms of the image of the transfer of the factors of the gluing.
Section~\ref{max_sec} is devoted to the maximal parabolic subgroups of a finite symplectic group and 
Section~\ref{class_groups} deals with finite singular classical groups. 
In Section~\ref{para_sec} we introduce parabolic gluing and compute the invariant field of  fractions for a range of representations.
We conclude with Section~\ref{diag_sec}, which introduces diagonal gluing and explores a number of examples.

\section{The Gluing Construction}\label{glue_constr}

Let $W_1$ and $W_2$ denote representations over a field $\F$ of groups $G_1$ and $G_2$ respectively.
The vector space of linear transformations $\Hom_{\F}(W_2,W_1)$ is a left $\F G_1$/right $\F G_2$ bimodule:
for $g_1\in G_1$, $\varphi\in\Hom_{\F}(W_2,W_1)$, $g_2\in G_2$ and $v\in W_2$ we have
$(g_1\cdot\varphi\cdot g_2)(v)=g_1\cdot(\varphi(g_2\cdot v))$.  
Using the unique unital ring homomorphism from $\Z$ to $\F$, every $\F G_i$-module is also a $\Z G_i$-module.
Let $\M$ denote  a left $\Z G_1$/right $\Z G_2$ sub-bimodule of $\Hom_{\F}(W_2,W_1)$.
We use $G_1\times_{\M} G_2$ to denote the semidirect product whose elements consist of triples
$(g_1,\varphi,g_2)\in G_1\times \M \times G_2$ with the product given by
$(g_1,\varphi,g_2)\cdot(g'_1,\varphi',g'_2)=(g_1g_1',g_1\varphi'+\varphi g_2',g_2g_2')$.
We refer to $G_1\times_{\M} G_2$  as the {\it gluing of $G_1$ to $G_2$ through $\M$}.
Note that, to perform this construction, we need $\M$ to be closed with respect to addition and with 
respect to the group actions. There is a natural action of $G_1\times_{\M} G_2$ on $V:=W_1\oplus W_2$
given by $(g_1,\varphi,g_2)(w_1\oplus w_2)=(g_1w_1+\varphi(w_2))\oplus g_2w_2$.
If $W_1$ and $W_2$ are faithful, then $V$ is a faithful  representation of $G_1\times_{\M} G_2$.
If we choose bases for $W_1$ and $W_2$ and denote the resulting matrices by $[g_1]$, $[\varphi]$ and $[g_2]$,
then the associated matrix group is given by
$$
\left\{\begin{pmatrix}
   [g_1]   &   [\varphi] \\
     0 &  [g_{2}]
\end{pmatrix}\in\GL_{m+n}(\F)\Big |~g_{1}\in G_{1},g_{2}\in G_{2},\varphi\in \M\right\}$$
where $m=\dim(W_1)$ and $n=\dim(W_2)$.
If $W_1$ and $W_2$ are both faithful, then the matrix group is isomorphic to $G_1\times_{\M} G_2$ .

Since $\M$ is a normal subgroup of $G_1\times_{\M} G_2$ with quotient isomorphic to $G_1\times G_2$,
we can compute the ring of invariants $\F[V]^{G_1\times_{\M} G_2}$ by first computing the invariants under the action
of $\M$ and then computing the invariants under the action of $G_1\times G_2$, in other words,
$\F[V]^{G_1\times_{\M} G_2}=(\F[V]^{\M})^{G_1\times G_2}$. We will routinely identify $\F[W_2]$ with the subalgebra
$\F\otimes\F[W_2]\subset \F[W_1]\otimes \F[W_2]=\F[W_1\oplus W_2]=\F[V]$. Using this identification, we observe that 
$\F[W_2]\subset \F[V]^{\M}$.
We will say that the gluing is {\it split} if there exists a subalgebra $A\subset\F[V]^{\M}$
such that
\begin{itemize}
\item $\F[V]^{\M}=A \otimes \F[W_2]$,
\item $A$ is an $\F(G_1\times G_2)$ submodule of  $\F[V]^{\M}$ and
\item $G_2$ acts trivially on $A$.
\end{itemize}
When we write $\F[V]^{\M}=A \otimes \F[W_2]$ above, we mean that there is an algebra isomorphism
from $A \otimes \F[W_2]$ to  $\F[V]^{\M}$ which restricts to the inclusion on each of the factors.
If the gluing is split, then
$\F[V]^{G_1\times_{\M} G_2}=A^{G_1}\otimes\F[W_2]^{G_2}$. 

We are primarily interested in rings of invariants for finite groups. If $\M$ is non-zero and $\F$ has characteristic zero, 
then $\M$ is an infinite group. For the remainder of the paper,
 we assume the characteristic of $\F$ is a prime number $p$ and that
$\M$ is finite. This means that $\M$ is an elementary abelian $p$-group and a vector space over the prime field $\F_p$.
Recent work on the rings of invariants of modular representation of elementary abelian $p$-groups includes \cite{CSW2013}, \cite{P2018} and \cite{PS2016}.

We will use $\{y_1,\ldots,y_m\}$ to denote a chosen basis for $W_1^*$ and 
$\{x_1,\ldots,x_n\}$ to denote a chosen basis for $W_2^*$. Observe that $\F[x_1,\ldots,x_n]\subset \F[V]^{\M}$.
We can use orbit products to construct additional $\M$-invariants: define
$N_{\M}(y_j):=\prod\{y_j\cdot g\mid g\in\M \}$. The set 
$$\HH:=\{x_1,\ldots,x_n,N_{\M}(y_1),\ldots, N_{\M}(y_m)\}$$ is a homogeneous system of parameters 
for $\F[V]^{\M}$ (see, for example, \cite[Example 6]{FS2016}).  
If $\HH$ is a generating set for $\F[V]^{\M}$, we say that the chosen basis is a {\it Nakajima basis} for the $\F\M$-module $V$.
Recall that $\mathcal{H}$ is a generating set for $\F[V]^{\mathcal{M}}$ if and only if the product of the degrees of the elements of $\mathcal{H}$
is equal to the order of $\mathcal{M}$ (see, for example, \cite[3.7.5]{DK2002}).
The polynomials $N_{\M}(y_j)$ have a number of special properties. The $\M$-orbit of $y_j$ is of the form
$y_j+U_j$ where $U_j$ is an $\F_p$-subspace of $W_2^*={\rm Span}_{\F}\{x_1,\ldots,x_n\}$. Thus
$$N_{\M}(y_j)=\prod_{u\in U_j} (y_j+u)=y_j^{p^{\ell_j}}+\sum_{k=0}^{\ell_j-1} d_{k-\ell_j,j} y^{p^k} $$
where $\ell_j=\dim_{\F_p}(U_j)$ and the $d_{i,j}$ are  Dickson polynomials associated to $U_j$ 
(see \cite{Wil1983} or \cite{Dic1911} ).
Therefore $N_{\M}(y_j)$ is additive as a function of $y_j$ and is invariant under any subgroup of $\GL(W_2)$ which stabilises $U_j$.

\begin{exam} \label{Fqexam}{\rm
Suppose $\F=\F_q$, where $q=p^r$ for some $r$, and take $\M=\Hom_{\F_q}(W_2,W_1)$.
Then, for all $j$, we have  $U_j=W_2^*$  and $\deg(N_{\M}(y_j))=q^n$.
Therefore, since the order of $\M$ is $q^{mn}$, we have
$\F_q[V]^{\M}=\F_q[x_1,\ldots,x_n,N_{\M}(y_1),\ldots, N_{\M}(y_m)]$.
Hence any basis consistent with the decomposition $V^*=W_1^*\oplus W_2^*$ is Nakajima.
Furthermore, any subgroup of $\GL(W_2)$ stabilises $W_2^*$. Thus $G_2$ acts trivially on
$A=\F_q[N_{\M}(y_1),\ldots, N_{\M}(y_m)]$ and so the gluing is split giving 
$\F_q[V]^{G_1\times_{\M}G_2}=A^{G_1}\otimes\F_q[W_1]^{G_2}$.
The algebra homomorphism from $\F_q[W_1]$ to $A$ which takes $y_j$ to $N_{\M}(y_j)$ is
a $G_1$-equivariant isomorphism which takes elements of degree $d$ to elements of degree $d\cdot q^n$.
This means that  $\F_q[V]^{G_1\times_{\M}G_2}$ is isomorphic to $\F_q[W_1]^{G_1}\otimes \F_q[W_2]^{G_2}$
by an isomorphism which is not degree preserving but does restrict to an isomorphism of the augmentation ideals.
}\end{exam}

\begin{defn}
We say that a gluing is {\bf polynomial} if there is a $(G_1\times G_2)$-equivariant isomorphism of $\F$-algebras
from $\F[W_1\oplus W_2] $ to $\F[V]^{\M}$. 
\end{defn}
In the case of a polynomial gluing,
$\F[V]^{G_1\times_{\M} G_2}$ is isomorphic to $\F[W_1]^{G_1}\otimes \F[W_2]^{G_2}$.
In most of the polynomial gluings we consider, the $\F$-algebra isomorphism $\psi: \F[W_1\oplus W_2] \to \F[V]^{\M}$
is an extension of the inclusion of $\F[W_2]$ into $\F[V]^{\M}$ and restricts to an isomorphism of
$\F[W_1]$ to a subalgebra $A\subset \F[V]^{\M}$;
in this case, the gluing is split and the induced map from $\F[W_1]$ to $A$  is $G_1$-equivariant.
Example~\ref{Fqexam} provides a canonical example of a family of split polynomial gluings.

\subsection*{Jia Huang's Gluing Lemma}
Jia Huang's Gluing Lemma \cite[\S 2]{Huang2011} is an extension of Example~\ref{Fqexam}.
We reformulate his result in our notation. 
Suppose $q=p^r$, $\F_q\subseteq\F$ and $M_1$ is an $\F_qG_1$-submodule of $W_1$.
Further suppose that $M_2$ is an $\F_qG_2$-submodule of $W_2$ with
$\F\cdot M_2=W_2$ and $\dim_{\F_q}(M_2)=\dim_{\F}(W_2)=n$. This means that every
$\F_q$-basis for $M_2$ is an $\F$-basis for $W_2$. Therefore, every element of
$\Hom_{\F_q}(M_2,M_1)$ extends uniquely to an element of $\Hom_{\F}(W_2,W_1)$,
giving an isomorphism from $\Hom_{\F_q}(M_2,M_1)$ to a left $\F_qG_1$/ right $\F_qG_2$ sub-bimodule
of $\Hom_{\F}(W_2,W_1)$;  take $\mathcal M$ to be the image of this map. 
The gluing $G_1\times_{\mathcal M} G_2$ is a split polynomial gluing.
To see this, first observe that $U_j=M_2^*$ for each $j$, which means that $N_{\mathcal M}(y_j)$
has degree $q^n$ and, as in Example~\ref{Fqexam}, 
$\F[V]^{\mathcal M}=\F[x_1,\ldots,x_n,N_{\mathcal M}(y_1),\ldots, N_{\mathcal M}(y_m)]$.
Then observe that, since any subgroup of $\GL(M_2)$ stabilises $M_2^*$,
$G_2$ acts trivially on $A=\F_q[N_{\M}(y_1),\ldots, N_{\M}(y_m)]$ .


\subsection*{Examples and Applications of Thin Gluing}
We say that a gluing is {\it thin} if either $\dim_{\F}(W_1)=1$ or  $\dim_{\F}(W_2)=1$.
If $\dim_{\F}(W_1)=1$, then the gluing is polynomial with $A=\F[N_{\M}(y_1)]$.
If $\dim_{\F}(W_2)=1$, then the representation of $\M$ is what is known as a hyperplane representation; 
the ring of invariants $\F[V]^{\M}$ is still a polynomial algebra but the 
generators are not necessarily orbit products of variables (see \cite{CC2012}, \cite{HS2008} and \cite{LS1987}) 

\begin{thm}\label{thin_glue} If a finite group $G$ acts on an elementary abelian $p$-group $E$  and $\F$ is a sufficiently large field of
characteristic $p$, then there is a faithful representation of the semidirect product $G \ltimes E$ 
of dimension $|G|+1$ over $\F$ .
\end{thm}

\begin{proof}
  The action of $G$ on $E$ makes $E$ a module over the group ring $\F_pG$. Every $\F_pG$ module has an injective envelope.
  Since injective $\F_pG$-modules are projective, this means we can embed $E$ into a free $\F_pG$-module, say
  $\mathcal{F}=\oplus_{i=1}^r\F_pGb_i$.
  Choose elements $c_1,\ldots,c_r\in \F$ so that $\{c_1,\ldots,c_r\}$ is linearly independent over $\F_p$. Then we can embed
  $\mathcal{F}$ into $\F G$ using the $\F_pG$-module map which takes $b_i$ to $c_i$. Composing the map from $E$ to $\mathcal{F}$ with
  the map from $\mathcal{F}$ to $\F G$ gives an $\F_pG$-module isomorphism from $E$ to an $\F_pG$-submodule of $\F G$, say $\mathcal{E}$.
  If we take $G_1=G$, $W_1=\F G$, $G_2={\bf 1}=\{1\}$ and $W_2=\F$, the thin gluing $G\times_{\mathcal{E}}{\bf 1}$ is isomorphic to
  $G\ltimes E$ and the associated representation is faithful with dimension $|G|+1$.
\end{proof}

\begin{coro} \label{min_dim}
  If the cyclic $p$-group of order $p^r$, $C_{p^r}$, acts on an elementary abelian $p$-group $E$ and
  $E$ contains a free $\F_pC_{p^r}$-submodule , then 
the minimum dimension of a faithful representation of $C_{p^r}\ltimes E$ over $\F$ is $p^r+1$, for $\F$ sufficiently large.
\end{coro}

\begin{proof}
  If $E$ contains a free $\F_pC_{p^r}$-submodule then $C_{p^r}\ltimes E$ contains an element of order $p^{r+1}$ (otherwise
  $C_{p^r}\ltimes E$ has exponent $p^r$).
  Since an element of $p$-power order in $\GL_{p^r}(\F)$ has order at most $p^r$,
  the minimum dimension of a faithful representation is at least $p^r+1$.
  If $\F$ is sufficiently large, then the existence of a faithful representation of dimension $p^r+1$ is given by Theorem~\ref{thin_glue}.
\end{proof}

\begin{exam}{\rm
  Specialise Example~\ref{Fqexam} by taking $W_2=\F_qG_2$ and $W_1=\F_qG_1$.
  Then $\M=\hom_{\F_q}(\F_qG_2,\F_qG_1)$ can be given the structure of a left module over $\F_q(G_1\times G_2)$ in a natural way.
  With this choice of action, $\M$ is the principal module generated by the linear map which maps $1\in \F_q G_2$ to $1\in \F_q G_1$
  and is zero on all other group elements. This means that $\M$ is isomorphic to $\F_q(G_1\times G_2)$.
  Using this observation, it is easy to verify that
  $(G_1\times G_2)\ltimes \F_q(G_1\times G_2)\cong G_1\times_{\M}G_2$.
  Therefore, the representation associated to  $G_1\times_{\M}G_2$ gives a faithful representation of
  $G\ltimes\F_qG$ of dimension $|G_1|+|G_2|$ whenever $G\cong G_1\times G_2$. Taking $G_1=G_2=C_p$ gives
  an upper bound of $2p$ on on the minimum dimension of a faithful representation of $(C_p\times C_p)\ltimes \F_p(C_p\times C_p)$.
} \end{exam}

\begin{rem} Theorem \ref{thin_glue} and Corollary \ref{min_dim} had their origins in a question from
  Alex Duncan related to his work with Christian Urech on
  representations of finite subgroups of Cremona groups.
\end{rem}

\section{Properties of Tensor Products}\label{tensor_sec}

In this section we summarise a number of properties of $\F[W_1]^{G_1}\otimes\F[W_2]^{G_2}$ which are inherited from $\F[W_1]^{G_1}$ and $\F[W_2]^{G_2}$.
We expect that most of the results of this section will be familiar to the experts.
Throughout we identify $\F[W_1]^{G_1}$ with $\F[W_1]^{G_1}\otimes\F\subset \F[W_1]^{G_1}\otimes\F[W_2]^{G_2}$ 
and $\F[W_2]^{G_2}$ with $\F\otimes \F[W_2]^{G_2}\subset \F[W_1]^{G_1}\otimes\F[W_2]^{G_2}$.
Let $\mathcal A$ denote the category of finitely generated commutative 
$\F$-algebras which are graded over the non-negative integers
and have both a unit and an augmentation. Most of the results of this section are valid for objects in this category.
One approach to this material would be to follow the model developed in \cite{BCM2018} for Noetherian local rings.
We have elected to take a more direct but ad hoc approach.

\begin{lem} \label{gen_lem}
Suppose $B$ and $C$ are objects in $\mathcal A$. Then
$$\dim_{\F}\left(\left(B\otimes C\right)_+/\left(B\otimes C\right)_+^2\right)=
\dim_{\F}(B_+/B_+^2)+\dim_{\F}(C_+/C_+^2).$$
\end{lem}

\begin{proof} Observe that 
$(B\otimes C)_+ = (B_+\otimes\F)  \oplus  (B_+\otimes C_+) \oplus (\F\otimes C_+)$ and
$ (B\otimes C)_+^2=(B_+^2\otimes \F) \oplus (\ B_+\otimes C_+) \oplus (\F\otimes C_+^2)$.
\end{proof}

Note that if $B$ is an object in $\mathcal A$, then every set of minimal homogeneous generators for $B$
projects to a homogeneous basis for the finite dimensional graded vector space $B_+/B_+^2$. 
Therefore, although $B$ does not have a unique minimal homogeneous generating set, the degrees of
a minimal homogeneous generating set are given by the dimensions of the homogeneous components of
$B_+/B_+^2$.  The following is a straightforward consequence of Lemma~\ref{gen_lem} and its proof.

\begin {prop} \label{Gen-prop}
If $\{f_1,\ldots,f_s\}$ is a minimal set of homogeneous generators for  $\F[W_1]^{G_1}$
and $\{h_1,\ldots, h_k\}$ is a minimal set of homogeneous generators for  $\F[W_2]^{G_2}$, then
$\{f_1,\ldots,f_s,h_1,\ldots,h_k\}$ is a minimal set of homogeneous generators for  
$\F[W_1]^{G_1}\otimes\F[W_2]^{G_2}$.
\end{prop} 

We will use Koszul complexes to relate properties of $B\otimes C$  to properties of $B$ and $C$.
We suggest \cite[\S 1.6]{BH1993} as a good reference for the basic properties of Koszul complexes.
Suppose $B$ is an object in $\mathcal{A}$ and $f_1,\ldots,f_s$ is a minimal set of homogeneous generators for $B$.
Let $K_B(\underline{f})$ 
denote the Koszul complex determined by the sequence $f_1,\ldots, f_s\in B$.
Similarly, let $K_C(\underline{h})$ denote the Koszul complex determined by a minimal set of homogeneous generators $h_1,\ldots, h_k$ for $C$.
We write $K_{B\otimes C}(\underline{f},\underline{h})$ for the Koszul complex determined by 
$f_1\otimes 1,\ldots,f_s\otimes 1,1\otimes h_1,\ldots,1\otimes h_k\in B\otimes C$.
\begin{lem}\label{Koszul-lem}
  There is an isomorphism of differential graded $(B\otimes C)$-algebras from
  $K_{B\otimes C}(\underline{f},\underline{h})$ 
to $K_B(\underline{f})\otimes K_C(\underline{h})$.
\end{lem} 

\begin{proof}
The differential graded algebra $K_{B\otimes C}(\underline{f},\underline{h})$
is the exterior algebra on the free $(B\otimes C)$-module generated
$\{e_i,b_j\mid i=1,\ldots,s;j=1,\ldots,k\}$ with the differential determined by
$d(e_i)=f_i\otimes 1$ and $d(b_j)=1\otimes h_j$. Using the univeral property of the exterior
algebra, the map taking $e_i$ to $e_i\otimes 1$ and $b_i$ to $1\otimes b_i$, extends to an algebra
map from  $K_{B\otimes C}(\underline{f},\underline{h})$ 
to $K_B(\underline{f})\otimes K_C(\underline{h})$.
Using the Koszul sign convention to define the differential on $K_B(\underline{f})\otimes K_C(\underline{h})$,
we see that the map is a map of differential graded algebras. Since the graded components of both
algebras are free $(B\otimes C)$-modules of the same rank and the map takes a basis to a basis, we have the
required isomorphism.
  \end{proof}

It follows from Lemma \ref{Koszul-lem} that the homology of $K_{B\otimes C}(\underline{f},\underline{h})$ 
is isomorphic to the homology of $K_B(\underline{f})\otimes K_C(\underline{h})$.
Since the tensor product is over a field, the K\"unneth formula \cite[Theorem 3.6.3]{Weibel1994} gives
$$H_n\left(K_B(\underline{f})\otimes K_C(\underline{h})\right)\cong \bigoplus_{i+j=n} H_i\left(K_B(\underline{f})\right)\otimes H_j\left(K_C(\underline{h})\right).$$
Since the Koszul homology is independent of the choice of generating set we write $H_*(B)$ for the homology of $K_B(\underline{f})$.
Using this convention and the above observations, we have
\begin{equation}\label{Koz_eq}
H_n\left(B\otimes C\right)\cong \bigoplus_{i+j=n} H_i\left(B\right)\otimes H_j\left(C\right).
\end{equation}

Recall that, for an object $B$ in $\mathcal{A}$, the depth of $B$ is 
$\grade(B_+,B)$, the length of a maximal $B$-regular sequence in $B_+$.
Since the grade can be computed using the Koszul complex,
see \cite[Theorem 1.6.17]{BH1993}, we have the following proposition.

\begin{prop} \label{Depth-prop}
$\depth(\F[W_1]^{G_1}\otimes\F[W_2]^{G_2})=\depth(\F[W_1]^{G_1})+\depth(\F[W_2]^{G_2})$.
\end{prop}

Since an object in $\mathcal{A}$ is Cohen-Macaulay if and only if the depth is equal to the Krull dimension,
we get the following.

\begin{prop} \label{CM-prop} 
$\F[W_1]^{G_1}\otimes\F[W_2]^{G_2}$ is Cohen-Macaulay if and only if
both $\F[W_1]^{G_1}$ and $ \F[W_2]^{G_2}$ are Cohen-Macaulay.
\end {prop}

Suppose $B$ is an object in $\mathcal{A}$ and $f_1,\ldots,f_s$ is a minimal set of homogeneous generators for $B$.
Denote $R:=\F[X_1,\ldots, X_s]$ and consider the resulting presentation $\rho: R\ra B$ given by $\rho(X_i)=f_i$.
We say that $B$ is a {\it complete intersection} if $\ker(\rho)$ is generated by a regular sequence.
We use $\dim(B)$ to denote the Krull dimension of $B$. The analogue of the
following result is well-known in the setting of Noetherian local rings, see for example \cite[\S 2.3]{BH1993} and \cite[\S 21]{Mat1986}.

\begin{prop}\label{CI-char} Suppose $B$ is an integral domain. Using the above notation, $B$ is a complete intersection
  if and only if  $\dim_{\F}(H_1(B))=s-\dim(B)$.
\end{prop}

\begin{proof}
Note that, since $B$ is an integral domain, $\ker(\rho)$ is a prime ideal.
Let $\mu(\ker(\rho))$ denote the number of elements in a minimal generating set for $\ker(\rho)$. 
If $B$ is a complete intersection then $\mu(\ker(\rho))=\grade(\ker(\rho),R)=\height(\ker(\rho))$.
Conversely, since $R$ is Cohen-Macaulay, $\height(\ker(\rho))=\grade(\ker(\rho),R)$ and if
$\grade(\ker(\rho),R)=\mu(\ker(\rho))$, then $\ker(\rho)$ is generated by a regular sequence, see \cite[Theorem~16.21]{Sharp2000}.
Observe that $K_B(\underline{f})\cong K_R(\underline{X})\otimes_R B$.
Therefore $H_1(K_B(\underline{f}))\cong \Tor^R_1(\F,B)$.
The short exact sequence of $R$-modules $0\to\ker(\rho)\to R \to B\to 0$
gives a long exact sequence ending in
$$\Tor^R_1(\F,R)\to \Tor^R_1(\F,B)\to \F\otimes_R \ker(\rho) \to \F\otimes_R R \to \F\otimes_R B \to 0.$$
Note that $\F\otimes_R R\cong\F\cong \F\otimes_R B$ and $\Tor^R_1(\F,R)=0$.
Hence
$$H_1(B)\cong \Tor^R_1(\F,B)\cong\F\otimes_R \ker(\rho)\cong \ker(\rho)/(R_+\ker(\rho))$$
and $\dim_{\F}(H_1(B))=\mu(\ker(\rho))$ (compare with \cite[page 170]{Mat1986}).
Recall that $\height(\ker(\rho))=\dim(R)-\dim(B)$ (see, for example, \cite[Theorem A.16]{BH1993}).
Puting these ideas together, we see that $B$ is a complete intersection if and only if
$\dim_{\F}(H_1(B))=\dim(R)-\dim(B)=s-\dim(B)$.
\end{proof}

\begin{prop}\label{CI-prop} 
$\F[W_1]^{G_1}\otimes\F[W_2]^{G_2}$ is a complete intersection  if and only if
both $\F[W_1]^{G_1}$ and $ \F[W_2]^{G_2}$ are complete intersections.
\end {prop}

\begin{proof}
  Since $\F[W_1]^{G_1}$, $\F[W_2]^{G_2}$ and $\F[W_1\oplus W_2]^{G_1\times G_2}\cong \F[W_1]^{G_1}\otimes\F[W_2]^{G_2}$
  are all integral domains, Proposition \ref{CI-char} applies. Using Equation \ref{Koz_eq},
  we have
  $$\dim_{\F}\left( H_1(\F[W_1]^{G_1}\otimes\F[W_2]^{G_2})\right)=\dim_{\F}\left(H_1(\F[W_1]^{G_1})\right)
  +\dim_{\F}\left(H_2(\F[W_2]^{G_2})\right).$$
 Suppose $\{f_1,\ldots,f_s\}$ is a minimal set of homogeneous generators for  $\F[W_1]^{G_1}$
and $\{h_1,\ldots, h_k\}$ is a minimal set of homogeneous generators for  $\F[W_2]^{G_2}$.
Using Proposition \ref{Gen-prop}, we see that
$\{f_1,\ldots,f_s,h_1,\ldots,h_k\}$ is a minimal set of homogeneous generators for  
$\F[W_1]^{G_1}\otimes\F[W_2]^{G_2}$. Therefore $\F[W_1]^{G_1}\otimes\F[W_2]^{G_2}$
is a complete intersection if and only if
$$\dim_{\F}\left(H_1(\F[W_1]^{G_1}\otimes\F[W_2]^{G_2})\right)=(s+k)-\dim(\F[W_1]^{G_1}\otimes\F[W_2]^{G_2}).$$
Since $G_1$ and $G_2$ are finite groups and $\F[W_1]^{G_1}\otimes\F[W_2]^{G_2}\cong\F[W_1\oplus W_2]^{G_1\times G_2}$, we have
$$\dim(\F[W_1]^{G_1}\otimes\F[W_2]^{G_2})=\dim_{\F}(W_1\oplus W_2)=\dim(\F[W_1]^{G_1})+\dim(\F[W_2]^{G_2})$$
and the result follows.  
\end{proof}  

\begin{prop} \label{UFD-prop}
$\F[W_1]^{G_1}\otimes\F[W_2]^{G_2}$ is a unique factorisation domain  if and only if
both $\F[W_1]^{G_1}$ and $ \F[W_2]^{G_2}$ are unique factorisation domains.
\end {prop}

\begin{proof} Again we identify $\F[W_1]^{G_1}\otimes\F[W_2]^{G_2}$ with $\F[W_1\oplus W_2]^{G_1\times G_2}$.
  Using a result of Nakajima (see \cite[Corollary 3.9.3]{Ben1993}), a ring of invariants of a finite group
  is a unique factorisation domain if and only if there are no non-trivial homomorphisms from the group
  to the units of the field taking the value one on every pseudoreflection. The pseudoreflections
  for the action of $G_1\times G_2$ on $W_1\oplus W_2$ are precisely the elements $(g_1,1)$ and
  $(1,g_2)$ for $g_1$ a pseudoreflection for the action of $G_1$ on $W_1$ and
  $g_2$  a pseudoreflection for the action of $G_2$ on $W_2$. A homomorphism
  from $G_1\times G_2$ to $\F^{\times}$
  is determined uniquely by the restrictions
  to $G_1\times\{1\}$ and  $\{1\}\times G_2$. Therefore, any homomorphism
  $\phi:G_1\times G_2\to\F^{\times}$ which takes value one on every pseudoreflection will
  restrict to give homomorphisms $\phi_1:G_1\to \F^{\times}$ and $\phi_2:G_2\to \F^{\times}$
  which take value one on every pseudoreflection.
   \end{proof}

For a representation $V$ of a finite group $G$ a subset $\{f_1,\ldots,f_s\}\subset\F[V]^G$
is called a {\it geometric separating set} if the elements of the set can be use to distinguish (separate)
the $G$-orbits of $V\otimes_{\F}\overline{\F}$, where $\overline{\F}$ is the algebraic closure of $\F$
(see, for example, \cite{DEK2009}).

\begin {prop} \label{Sep-prop}  If $\{f_1,\ldots,f_s\}$ is a homogeneous geometric separating set for  $\F[W_1]^{G_1}$
and $\{h_1,\ldots, h_k\}$ is a homogeneous geometric separating set for  $\F[W_2]^{G_2}$, then
$\{f_1,\ldots,f_s,h_1,\ldots,h_k\}$ is a homogeneous geometric separating set for  
$\F[W_1]^{G_1}\otimes\F[W_2]^{G_2}$.
\end{prop} 

\begin{proof}
  We identify $\F[W_1]^{G_1}\otimes\F[W_2]^{G_2}$ with $\F[W_1\oplus W_2]^{G_1\times G_2}$.
  Let $B_1$ denote the subalgebra of $\F[W_1]$ generated by $\{f_1,\ldots,f_s\}$ and
  let $B_2$ denote the subalgebra of $\F[W_2]$ generated by $\{h_1,\ldots,h_k\}$.
  Then $B:=B_1\otimes B_2$ is the subalagebra of  $\F[W_1\oplus W_2]$ generated by $\{f_1,\ldots,f_s,h_1,\ldots,h_k\}$.
  Using \cite[Proposition~1.2]{DEK2009}, $\F[W_i]^{G_i}$ is the purely inseparable closure of $B_i$ in $\F[W_i]$.
  Let $\overline{B}$ denote the purely inseparable closure of $B$ in $\F[W_1\oplus W_2]$.
  Since $B\subseteq\F[W_1\oplus W_2]^{G_1\times G_2}$, we see that $\overline{B}\subseteq\F[W_1\oplus W_2]^{G_1\times G_2}$.
  Therefore, we need only show that every homogeneous element
  of  $\F[W_1\oplus W_2]^{G_1\times G_2}$ lies in $\overline{B}$. Consider $f\in\F[W_1\oplus W_2]^{G_1\times G_2}$.
  Using Proposition~\ref{Gen-prop}, write $f=\sum_{j=1}^{\ell}\alpha_j\otimes \beta_j$ for some choice of
  $\alpha_j\in\F[W_1]$ and $\beta_j\in\F[W_2]$. Choose $N\in\Z^+$ so that $\alpha_j^{p^N}\in B_1$ and $\beta_j^{p^N}\in B_2$
  for all $j$. Then, since taking the $p^N$-power is additive, we have
  $f^{p^N}=\sum_{j=1}^{\ell}\alpha_j^{p^N}\otimes \beta_j^{p^N}\in B_1\otimes B_2=B$, as required.
  \end{proof}

\begin{thm}\label{prop_thm} Suppose $G_1\times_{\mathcal M} G_2$ is a polynomial gluing. Then
\begin{itemize} 
\item[(i)] $\F[V]^{G_1\times_{\mathcal M} G_2}$ is polynomial if and only if
both $\F[W_1]^{G_1}$ and $ \F[W_2]^{G_2}$ are polynomial;
\item[(ii)]
$\F[V]^{G_1\times_{\mathcal M} G_2}$ is Cohen-Macaulay if and only if
both $\F[W_1]^{G_1}$ and $ \F[W_2]^{G_2}$ are Cohen-Macaulay;
\item[(iii)]
$\F[V]^{G_1\times_{\mathcal M} G_2}$ is a complete intersection if and only if
both $\F[W_1]^{G_1}$ and $ \F[W_2]^{G_2}$ are complete intersections;
\item[(iv)]
$\F[V]^{G_1\times_{\mathcal M} G_2}$ is unique factorisation domain if and only if
both $\F[W_1]^{G_1}$ and $ \F[W_2]^{G_2}$ are unique factorisation domains.
\end{itemize}
\end{thm}
\begin{proof} Each of these properties is preserved by the gluing isomorphism. Therefore the results follow from 
Propositions \ref{Gen-prop}, \ref{CM-prop}, \ref{CI-prop} and \ref{UFD-prop}.
\end{proof}

\begin{thm}\label{Split_thm} Suppose $G_1\times_{\mathcal M} G_2$ is a split polynomial gluing. Then
$$\depth(\F[V]^{G_1\times_{\mathcal M} G_2})=\depth(\F[W_1]^{G_1})+\depth(\F[W_2]^{G_2}). $$
Furthermore, if $\psi: \F[W_1\oplus W_2] \to \F[V]^{\M}$ denotes an $\F$-algebra isomorphism which
is an extension of the inclusion of $\F[W_2]$ into $\F[V]^{\M}$ and restricts to a $G_1$-equivariant isomorphism of
 $\F[W_1]$ to $A$,
 then
$\{\psi(f_1),\ldots,\psi(f_s),h_1,\ldots,h_k)\}$ is a geometric separating set for 
$\F[V]^{G_1\times_{\mathcal M} G_2}$ where
$\{f_1,\ldots,f_s\}$ is a homogeneous geometric separating set for  $\F[W_1]^{G_1}$
and $\{h_1,\ldots, h_k\}$ is a homogeneous geometric separating set for  $\F[W_2]^{G_2}$.
\end{thm}
\begin{proof} 
While the map $\psi$ is not degree preserving, it maps $\F[W_1\oplus W_2]_+ ^{G_1\times G_2}$
to $\F[V]^{G_1\times_{\mathcal M} G_2}_+$. Therefore, the results follow from Propositions \ref{Depth-prop} and \ref{Sep-prop}.
\end{proof} 

\section{The Image of the Transfer}\label{trans_sec}

Suppose $G:=G_{1}\times_{\M}G_{2}$ is a polynomial gluing with
gluing isomorphism $\psi: \F[V]\ra\F[V]^{\M}$ where $\F[V]=\F[W_{1}\oplus W_{2}]$. 
We continue to identify $\F[W_{i}]$ with the appropriate subalgebra of $\F[V]$.
Recall that the transfer map $\Tr^G:\F[V]\ra\F[V]^G$ is the
morphism of $\F[V]^G$-modules defined by $\Tr^G(f):=\sum_{g\in G} f\cdot g$.
The image of the transfer is an ideal in $\F[V]^G$.
By a result of Bram Broer \cite{Broer2005}, since $\M$ is a $p$-group and $\F[V]^{\M}$ is polynomial,
the image of $\Tr^{\M}$ is a principal ideal in $\F[V]^{\M}$.
Let $\tau$ denote a generator for this ideal.
There is a factorisation
$\Tr^G=\Tr^{G_1\times G_2}\circ\Tr^{\M}$.
Furthermore, since $\M$ is a normal subgroup of $G$, 
 $G_1\times G_2$ stabilises the image of $\Tr^{\M}$; to see this observe that $(\sum_{h\in\M}f\cdot h)g=\sum_{h\in\M}(f\cdot g)(g^{-1}hg)$.
Therefore, for any $g\in G_1\times G_2$, we see that $\tau\cdot g$ is a scalar multiple of $\tau$. Thus we define a character of $G_1\times G_2$ by
$\tau\cdot g=\chi_{\tau}(g)\tau$. If $G_1\times G_2$ is a $p$-group, then $\chi_{\tau}$ is trivial and $\tau\in\F[V]^{G_1\times G_2}$.

\begin{prop}\label{ptransfer}
  Suppose $\{u_1,\ldots u_k\}$ is a generating set for the image of $\Tr^{G_1}$, $\{v_1,\ldots,v_s\}$ is a generating set for the image of $\Tr^{G_2}$
  and $\tau\in\F[V]^{G_1\times G_2}$ is a generator for the image of $\Tr^{\M}$.
Then the image of $\Tr^G$ is the ideal generated by $\{\tau\psi(u_iv_j)\mid 1\leqslant i\leqslant k,1\leqslant j\leqslant s\}$. 
\end{prop}

We give the proof of Proposition \ref{ptransfer} after proving Lemma \ref{tr_lem} below.

\begin{exam}{\rm
  In the context of Example~\ref{Fqexam}, it follows from \cite[Theorem~4.4]{SW-trans} that $\tau=d_{n,n}(W_2^*)^m\in\F[V]^{G_1\times G_2}$.
  Hence, in this context, we can apply Proposition~\ref{ptransfer} to compute the image of $\Tr^G$ in terms of the image of $\Tr^{G_1}$
  and the image of $\Tr^{G_2}$.
  
  For instance, let $\mathcal{U}(n,\F_{p})$ be the group of $n \times n$ upper-triangular unipotent matrices over $\F_{p}$ and
consider $G=\mathcal{U}(4,\F_{p})$ and $G_{1}=G_{2}=\mathcal{U}(2,\F_{p})$. Then $G= G_{1}\times_{\M}G_{2}$ with $\M=\Hom_{\F_p}(\F_p^2,\F_p^2)$. 
By  \cite[Theorem~4.4]{SW-trans} we see that the image of $\Tr^{G}$ is the principal ideal generated by 
$$\delta:=d_{1,1}(x_{1})\cdot d_{2,2}(x_{1},x_{2})\cdot d_{3,3}(x_{1},x_{2},y_{1})$$
where $d_{i,j}$ denotes the Dickson invariants in the specified variables. In particular, $d_{1,1}(x_{1})=x_{1}^{p-1}$,
$d_{2,2}(x_{1},x_{2})=\det\Big(\begin{smallmatrix}
  x_{1}   & x_{1}^{p}   \\
   x_{2}  &  x_{2}^{p}
\end{smallmatrix}\Big)^{p-1}$ and 
$$d_{3,3}(x_{1},x_{2},y_{1})=\det\begin{pmatrix}
   x_{1}   & x_{1}^{p} &x_{1}^{p^{2}}  \\
    x_{2}  &  x_{2}^{p}& x_{2}^{p^{2}}\\
    y_{1} & y_{1}^{p} & y_{1}^{p^{2}}
\end{pmatrix}^{p-1}.$$
On the other hand, applying \cite[Theorem~4.4]{SW-trans} again we observe that the image of $\Tr^{\M}$ is the principal ideal generated by
$\tau=d_{2,2}(x_{1},x_{2})^{2}$,
which is $G_{1}\times G_{2}$-invariant. The image of $\Tr^{G_{1}}$ is generated by
$u:=d_{1,1}(x_{1})=x_{1}^{p-1}$ and the image of $\Tr^{G_{2}}$ is generated by
$v:=d_{1,1}(y_{1})=y_{1}^{p-1}$. Note that 
$\psi(u)=d_{1,1}(x_{1})$ and 
$$\psi(v)=\psi(y_{1})^{p-1}=(y_{1}^{p^{2}}+d_{1,2}(x_{1},x_{2})y_{1}^{p}+d_{2,2}(x_{1},x_{2})y_{1})^{p-1}.$$
Using Proposition 1.3(b) of \cite{Wil1983}, we see that $d_{3,3}(x_1,x_2,y_1)=-d_{2,2}(x_1,x_2)\psi(v)$.
Therefore $\tau\cdot \psi(u)\cdot\psi(v) =d_{2,2}(x_{1},x_{2})^{2} \cdot d_{1,1}(x_{1})\cdot\psi(v) 
=-d_{1,1}(x_{1})\cdot d_{2,2}(x_{1},x_{2}) \cdot d_{3,3}(x_{1},x_{2},y_{1})=-\delta$.
}\end{exam}

\begin{lem}\label{tr_lem}
  Suppose $\{u_1,\ldots u_k\}$ is a generating set for the image of $\Tr^{G_1}$ and $\{v_1,\ldots,v_s\}$ is a generating set for the image of $\Tr^{G_2}$.
  Then the image of $\Tr^{G_{1}\times G_{2}}$ is generated by $\{u_{i}v_{j}\mid 1\leqslant i\leqslant k,1\leqslant j\leqslant s\}$.
\end{lem}

\begin{proof}
  Let $I^{G_1\times G_2}$ denote the image of $\Tr^{G_1\times G_2}$, let $I^{G_1}$ denote the image of $\Tr^{G_1}$ and
  let $I^{G_2}$ denote the image of $\Tr^{G_2}$.
  Observe that $\Tr^{G_1\times G_2}=\Tr^{\{1\}\times G_2}\circ\Tr^{G_1\times \{1\}}$.
  Choose $f_i\in \F[W_1]$ and $h_j\in\F[W_2]$ so that $\Tr^{G_1}(f_i)=u_i$ and $\Tr^{G_2}(h_j)=v_j$.
  Then 
  $$\Tr^{G_1\times G_2}(f_ih_j)=\Tr^{\{1\}\times G_2}\circ\Tr^{G_1\times \{1\}}(f_ih_j)=\Tr^{\{1\}\times G_2}(u_ih_j)=u_iv_j\in I^{G_1\times G_2}.$$
  On the other hand, for any $f\in\F[V]$, write $f=\sum_{I,J}c_{I,J}x^Iy^J$ for exponent sequences $I$ and $J$ with $c_{I,J}\in\F$.
  Then 
  $$\Tr^{G_1\times G_2}(f)=\sum_{I,J}c_{I,J}\Tr^{G_1\times G_2}(x^Iy^J)=\sum_{I,J}c_{I,J}\Tr^{G_2}(x^I)\Tr^{G_1}(y^J)\in I^{G_1}I^{G_2}.$$
  Hence $\Tr^{G_1\times G_2}(f)$ is in the ideal generated by $\{u_{i}v_{j}\mid 1\leqslant i\leqslant k,1\leqslant j\leqslant s\}$.
\end{proof}

\begin{proof}[Proof of Proposition \ref{ptransfer}]
  Let $I^G$ denote the image of $\Tr^G$.
 Choose $f_i\in \F[W_1]$ and $h_j\in\F[W_2]$ so that $\Tr^{G_1}(f_i)=u_i$ and $\Tr^{G_2}(h_j)=v_k$.
 Choose $\alpha\in\F[V]$ so that $\Tr^{\M}(\alpha)=\tau$.
 Then 
 \begin{eqnarray*}
 \Tr^G(\psi(f_ih_j)\alpha)&=&\Tr^{G_1\times G_2}\circ\Tr^{\M}(\psi(f_ih_j)\alpha)=\Tr^{G_1\times G_2}(\psi(f_ih_j)\tau)\\
 &=&\tau\psi(\Tr^{G_1\times G_2}(f_ih_j))=\tau\psi(u_iv_j)\in I^G.
 \end{eqnarray*}
 Suppose $f=\Tr^G(f')$. Then $f=\Tr^{G_1\times G_2}(\Tr^{\M}(f'))=\Tr^{G_1\times G_2}(\widetilde{f}\tau)$ for some $\widetilde{f}\in\F[V]^{\M}$.
 Since $\tau\in\F[V]^{G_1\times G_2}$ and $\psi$ is a $(G_1\times G_2)$-equivariant isomorphism, we get
 $$f=\tau\Tr^{G_1\times G_2}(\psi(f''))=\tau\psi(\Tr^{G_1\times G_2}(f''))$$
  for some $f''\in\F[V]$. Therefore
     $f$ is in the ideal generated by $\{\tau\psi(u_iv_j)\mid 1\leqslant i\leqslant k,1\leqslant j\leqslant s\}$. 
\end{proof}

\section{Maximal Parabolic Subgroups of Finite Symplectic Groups} \label{max_sec}

For this section, we let $V$ denote the defining representation of the symplectic group $\Sp_{2m}(\F_q)$ with $q=p^r$.
We choose an ordered basis  $e_1,e_2,\ldots, e_m,f_m,f_{m-1},\ldots,f_1$ for $V$ 
with dual basis $y_1,\ldots,y_m,x_m,\ldots,x_1$ so that the symplectic form is represented by the matrix
$$
J=\begin{pmatrix}
\phantom{-} 0 &Q\\
-Q& 0
\end{pmatrix}
\, {\rm with\,} \, 
Q=\begin{pmatrix}
0&0 &0&\cdots&0& 0 &1\\
0&0 &0&\cdots&0& 1 & 0\\
&&&\vdots&&&\\
0&1&0&\cdots&0&0&0\\
1& 0&0&\cdots&0&0 & 0
\end{pmatrix}.
$$
We identify $\Sp_{2m}(\F_q)$ with the set of matrices $A\in \GL_{2m}(\F_q)$ satisfying $A^TJA=J$.
Define $\xi_i:=y_m^{q^i}x_m-y_mx^{q^i}_m +\cdots + y_1^{q^i}x_1-y_1x_1^{q^i}$.
Carlisle and Kropholler proved that $\F_q[V]^{\Sp_{2m}(\F_q)}$ is the complete intersection generated by
$\xi_1,\ldots,\xi_{2m-1}$ and the Dickson invariants $d_{1,2m},\ldots,d_{m,2m}$, see \cite[Theorem~8.3.11]{Ben1993}.
Note that for $i>1$, $\xi_i$ can be constructed by applying Steenrod operations to $\xi_1$. 
Therefore $\Sp_{2m}(\F_q)$ is the subgroup of $\GL_{2m}(\F_q)$ which fixes $\xi_1$.

The symplectic group $\Sp_{2m}(\F_q)$ has a $BN$-pair of type $C_m$ 
(see \cite{Taylor1992} or \cite[\S 1.11]{Carter1985}).
To construct a maximal parabolic subgroup for $\Sp_{2m}(\F_q)$, we remove one of the $m$ generating reflections from the 
associated Weyl group, $N/(B\cap N)$, and lift the resulting subgroup to get a subgroup of $N$, say $\widetilde{N}$. 
The resulting parabolic subgroup is generated by $B$ and $\widetilde{N}$.
We label the vertices of the Coxeter graph so that the edge of weight $4$ joins vertex $m-1$ and $m$.
We will refer to the maximal parabolic subgroup constructed by removing the reflection corresponding to vertex $k$ as 
a maximal parabolic of type $k$ and denote this subgroup by $\maxpara_k$. 
The primary goal of this section is to compute $\F_q[V]^{\maxpara_k}$. In the following, we use $Q_k$ to represent the
$k\times k$ submatrix of $Q$ constructed by removing the first $m-k$ columns and the last $m-k$ rows.
Therefore 
$$J=\begin{pmatrix}
\phantom{-} 0&0&0&Q_k\\
\phantom{-} 0&0&Q_{m-k}&0\\
\phantom{-} 0&-Q_{m-k}&0&0\\
-Q_k&0&0&0
\end{pmatrix}. $$
Let $P_k$ denote the subgroup of $\GL_{2m}(\F_q)$ consisting of matrices of the form
$$\begin{pmatrix}
I_k & C_1& C_2&A\\
0 &I_{m-k}&0&B_1\\
0&0&I_{m-k}&B_2\\
0&0&0&I_k
\end{pmatrix}$$
subject to the relations
$C_1=-Q_kB_2^TQ_{m-k}$, $C_2=Q_kB_1^TQ_{m-k}$ and
$ A^TQ_k-Q_kA=B_2^TQ_{m-k}B_1-B_1^TQ_{m-k}B_2$.
Observe that $P_k$ is a $p$-group and a subgroup of $\Sp_{2m}(\F_q)$ .

\begin{prop} \label{max_para_prop} The maximal parabolic subgroup $\maxpara_k$ is isomorphic to 
$$\left(\GL_k(\F_q)\times \Sp_{2m-2k}(\F_q)\right)\ltimes P_k$$
with the action of $\GL_k(\F_q)\times \Sp_{2m-2k}(\F_q)$ on $P_k$
given by mapping $(A,B)\in \GL_k(\F_q)\times \Sp_{2m-2k}(\F_q)$ 
to
$$ \begin{pmatrix}
A &0&0\\
0&B&0\\
0&0& Q_k(A^{-1})^TQ_k
\end{pmatrix} \in
\Sp_{2m}(\F_q).$$
\end{prop}
\begin{proof} We use the BN-pair for $\Sp_{2m}(\F_q)$ described in Chapter 8 of Taylor's book \cite{Taylor1992}.
  The Borel subgroup $B$ is the group of upper triangular symplectic matrices and $N$ is the group
  of symplectic monomial matrices. 
  The Weyl group is generated by reflections $w_1,\ldots,w_m$.
  For $i<m$, $w_i$ can be lifted to the element $n_i\in N$
which takes $e_i$ to $-e_{i+1}$, $e_{i+1}$ to $e_i$, $f_i$ to $-f_{i+1}$, $f_{i+1}$ to $f_i$ and fixes the other basis vectors. 
The generator $w_m$ lifts to $n_m\in N$ which takes  $e_m$ to $-f_m$, $f_m$ to $e_m$ and fixes the other basis elements.
The maximal parabolic subgroup $\maxpara_k$ is generated by $B$ and $\{n_1,\ldots,n_m\}\setminus \{n_k\}$.

Direct calculation verifies that $P_k$ is precisely the set of symplectic matrices of the given block-form and
that the embedding of $\GL_k(\F_q)\times \Sp_{2m-2k}(\F_q)$ gives the set of symplectic matrices of 
that corresponding block-form. A further explicit calculation verifies that the embedding of $\GL_k(\F_q)\times \Sp_{2m-2k}(\F_q)$ 
normalises $P_k$. The removal of vertex $k$ from the Coxeter graph gives a graph of type $A_{k-1}\times C_{m-k}$.
The removal of the generator $n_k$ separates  the basis vectors into three sets:
$\{e_1,\ldots,e_k\}$, $\{e_{k+1} ,\ldots, e_m,f_m,\ldots, f_{k+1}\}$, $\{f_k,\ldots,f_1\}$. 
We are left with the usual $BN$-pair for $\GL_k(\F_q)$ on the span of $\{e_1,\ldots,e_k\}$, the usual $BN$-pair for
$\Sp_{2m-2k}(\F_q)$ on the span of $\{e_{k+1} ,\ldots, e_m,f_m,\ldots, f_{k+1}\}$, the action of $\GL_k(\F_q)$ on
$\{f_k,\ldots,f_1\}$ determined by the symplectic condition, extended by the normal subgroup $P_k$ to recover the Borel subgroup.
\end{proof}

Let $U_k$ denote the span of $\{e_1,\ldots,e_k\}$ and let $\Sp_{2m}(\F_q)_{U_k}$ denote the pointwise stabiliser subgroup.

\begin{coro} \label{para_cor}The maximal parabolic subgroup $\maxpara_k$ is isomorphic to 
$$\GL_k(\F_q)\ltimes \Sp_{2m}(\F_q)_{U_k}$$
with the action of $\GL_k(\F_q)$ as described in Proposition~\ref{max_para_prop}.
\end{coro}
\begin{proof} With our chosen basis, the matrices representing elements of $\Sp_{2m}(\F_q)_{U_k}$ are of the form
  $$\begin{pmatrix}
    I_k& A& B\\
    0  & C & D\\
    0  & E & F
  \end{pmatrix}$$
  where $\{A,E\}\subset\F_q^{k\times(2m-2k)}$, $\{B,F\}\subset\F_q^{k\times k}$, $C\in\F_q^{(2m-2k)\times(2m-2k)}$ and $D\in\F_q^{(2m-2k)\times k}$.
  The symplectic condition forces $E=0$, $F=I_k$ and $C\in\Sp_{2m-2k}(\F_q)$.
  If we restrict to $C=I_{2m-2k}$ and apply the symplectic condition,
  we recover $P_k$. Since the action of $\Sp_{2m-2k}(\F_q)$ normalises $P_k$, we see that
  $\Sp_{2m}(\F_q)_{U_k}$
  is isomorphic to $\Sp_{2m-2k}(\F_q)\ltimes P_k$, and the result follows from Proposition~\ref{max_para_prop}.
\end{proof}

\begin{notation} 
  For the rest of this section, we use $\widetilde{d}_{i,\ell}$ to denote the $i^{th}$ Dickson invariant in
  the first $\ell$ variables
taken from the ordered list $x_1,\ldots,x_m,y_m,\ldots,y_1$.
For example, $\widetilde{d}_{i,m}$ is the $i^{th}$ Dickson invariant in $x_1,\ldots, x_m$.
Let $W_k$ denote the span of $\{x_1,\ldots,x_m,y_m,\ldots, y_1\}\setminus\{y_1,\ldots,y_k\}$ and define 
$$N_{k}(t):=\prod_{v\in W_k}(t+v)=\sum_{j=0}^{2m-k}t^{p^{2m-k-j}}\widetilde{d}_{j,2m-k}\in\F_q[V][t].$$
Note that $N_k(x_i)=0$ and $N_k(y_i)=0$ if $i>k$.
\end{notation}

\begin{thm}\label{stab_sub_inv}
The ring of invariants $\F_q[V]^{\Sp_{2m}(\F_q)_{U_k}}$ is generated by 
$x_1,\ldots,x_k$, $\xi_1,\ldots,\xi_{2m-1}$, $N_k(y_1),\ldots, N_k(y_k)$, 
and $\widetilde{d}_{i,2m-k}$ for $i=1,2,\ldots,2m-k$.
Furthermore, this ring is Cohen-Macaulay and a complete intersection.
\end{thm}

\begin{proof} Since $\F_q[V]^{\Sp_{2m}(\F_q)}$ is Cohen-Macaulay and a complete intersection, it follows from
  \cite[Theorem B]{Kem2002} that $\F_q[V]^{\Sp_{2m}(\F_q)_{U_k}}$ is Cohen-Macaulay and a complete intersection.

  We will use Kemper's algorithm based on \cite[Theorem 2.7]{Kem2002} to compute a generating set
  for $\F_q[V]^{\Sp_{2m}(\F_q)_{U_k}}$.
  Given a $G$-invariant polynomial, Kemper's algorithm produces a set of $G_U$-invariant polynomials.
  The algorithm uses only the input polynomials and the subspace $U$.
  Applying the algorithm to $\{d_{i,2m}\mid i=1,\ldots, 2m\}$ produces a generating set for
  $\F_q[V]^{\GL_{2m}(\F_q)_{U_k}}$. By comparing the order of  $\GL_{2m}(\F_q)_{U_k}$ with product of the degrees,
  we see that $\F_q[V]^{\GL_{2m}(\F_q)_{U_k}}$ is the polynomial algebra generated by
  $\{N_k(y_i),\widetilde{d}_{j,2m-k}\mid i=1,\ldots,k;j=1,2,\ldots,2m-k\}$.
  Applying the algorithm to the $\xi_j$ produces $\{x_1,\ldots,x_k,\xi_1,\ldots,\xi_{2m-1}\}$.
 We do not claim that this is a minimal generating set.
\end{proof}

For $k=m$, we have
$$\Sp_{2m}(\F_q)_{U_m}=P_m=\left\{
\begin{pmatrix} 
I_m & A\\ 0& I_m \end{pmatrix} 
\mid A=Q_mA^TQ_m\right\},
$$
which is an elementary abelian $p$-group of order $q^{(m^2+m)/2}$.

\begin{coro}\label{eapg-cor} The ring of invariants $\F_q[V]^{P_m}$ is the complete intersection generated by
$x_1,\ldots,x_m$, $\xi_1,\ldots, \xi_{m-1}$, and $N_m(y_1),\ldots, N_m(y_m)$.
\end{coro}
\begin{proof}
  By Theorem \ref{stab_sub_inv},  $\F_q[V]^{P_m}$ is the complete intersection generated by
  $x_1,\ldots,x_m$, $\xi_1,\ldots,\xi_{2m-1}$, $N_m(y_1),\ldots, N_m(y_m)$, 
  and $\widetilde{d}_{i,m}$ for $i=1,2,\ldots,m$.
  Since  $\widetilde{d}_{i,m}\in\F_q[x_1,\ldots,x_m]$, these generators are redundant.
  An explicit calculation gives
  $$\xi_m=\left(\sum_{j=1}^m x_jN_m(y_j)\right)-\left(\sum_{i=1}^{m-1} \xi_id_{m-i,m}\right).$$
  Thus $\xi_m$ is also a redundant generator.
  Similar relations can be constructed to eliminate $\xi_j$ for $j>m$;
  one approach to doing this is to apply Steenrod operations to the relation for $\xi_m$.
  \end{proof}

\begin{rem} The above result is essentially Theorem 5.4.8 of \cite{Hussain2011}.
\end{rem}
  
The action of $\GL_k(\F_q)$ on $V$ is the direct sum of the action as a vector and a covector
on ${\rm Span}(e_1,\ldots,e_k)$ and ${\rm Span}(f_k,\ldots,f_1)$, in the sense of
\cite{BK2011} or \cite{CW2017}, and a trivial action on the remaining basis vectors.
This means that we can use the results of
  \cite{CW2017} to compute generators for $\F_q[V]^{\GL_k(\F_q)}$. Our approach to the construction
  of generators for $\F_q[V]^{\maxpara_k}=\left(\F_q[V]^{\Sp_{2m}(\F_q)_{U_k}}\right)^{\GL_k(\F_q)}$ is to
  apply a ``diagonal gluing'' to the invariants of  $\F_q[V]^{\GL_k(\F_q)}$.
  It is useful to start with the case $k=m$, which is closely related to the Sylow $p$-subgroup of
  $\Sp_{2m}(\F_q)$.
We use $USp_{2m}$ to denote the group of upper-triangular unipotent symplectic matrices,
 a Sylow p-subgroup for $\Sp_{2m}(\F_q)$. In the following theorem, 
we use $N(x_i)$ to denote the product of $x_i+v$ as $v$ runs over the span of $x_1,\ldots,x_{i-1}$.

\begin{thm} \label{sylow_thm}The ring of invariants $\F_q[V]^{USp_{2m}}$ is the complete intersection generated by\\
$\xi_1,\ldots,\xi_{2m-2}$, $N(x_1),\ldots,N(x_m)$ and $N_m(y_m),N_{m-1}(y_{m-1}),\ldots, N_1(y_1)$.
\end{thm}
\begin{proof} Let $H$ denote the upper triangular unipotent subgroup of $\GL_m(\F_q)$. 
Observe that $USp_{2m}$ is isomorphic to $H\ltimes P_m$ with the action given by restricting the action of $\GL_m(\F_q)$ 
given in Proposition \ref{max_para_prop}
to the subgroup $H$. Furthermore, this restriction gives a vector/covector action of $H$ on $V$.
  Bonnaf\'e and Kemper \cite{BK2011} showed that $\F_q[V]^H$ is a complete intersection generated by
  $N_H(y_i)$, $N(x_i)$ for $i=1,\ldots,m$ and additional invariants which they denote by $u_j$ for
  $j=2-m,\ldots,m-2$. Let $S$ denote the algebra generated by $x_1,\ldots,x_m$, $N_m(y_1),\ldots,N_m(y_m)$.
  Referring to Corollary \ref{eapg-cor}, we see that $\{x_1,\ldots,x_m$, $N_m(y_1),\ldots,N_m(y_m)\}$ 
  is a homogeneous system of parameters for  $\F_q[V]^{P_m}$ and that $\F_q[V]^{P_m}$
  is a finitely generated free $S$-module. Furthermore, the $S$-module generators can be taken to be monomials in the $\xi_i$.
  Therefore, since the $\xi_i$ are $H$-invariant, adjoining the $\xi_i$ to a generating set for $S^H$ gives a generating set for
  $\F_q[V]^{USp_{2m}}=(\F_q[V]^{P_m})^H$.
  The algebra homomorphism from $\F_q[V]$ to $S$ which fixes $x_i$ and maps $y_i$ to $N_m(y_i)$ is an $H$-equivariant
  isomorphism of algebras. Thus we can construct generators for $S^H$ by substituting $N_m(y_i)$ for $y_i$ in the generators for
  $\F_q[V]^H$. Note that $N_H(N_m(y_i))=N_i(y_i)$. Substituting $N_m(y_i)$ for $y_i$  in $u_j$ gives 
  \begin{eqnarray*}
  \widetilde{u}_0&=&\sum_{i=1}^m x_iN_m(y_i)\\
  \widetilde{u}_j&=&\sum_{i=1}^m x_i^{q^j} N_m(y_i)\\
  \widetilde{u}_{-j}&=&\sum_{i=1}^m x_iN_m(y_i)^{q^j}
  \end{eqnarray*}
  for $j=1,\ldots,m$. These invariants can be written in terms of the $\xi_i$ and the Dickson invariants in $x_1,\ldots,x_m$:
   \begin{eqnarray*}
  \widetilde{u}_0&=&\sum_{i=1}^m \xi_i\widetilde{d}_{m-i,m}\\
  \widetilde{u}_j&=&\sum_{i=0}^{m-j-1} \xi_{m-i-j}^{q^j} \widetilde{d}_{i,m}-\sum_{i=m-j+1}^m\xi_{i+j-m}^{q^{m-i}}\widetilde{d}_{i,m}\\
  \widetilde{u}_{-j}&=&\sum_{i=0}^m \xi_{m+j-i}\widetilde{d}_{i,m}^{q^j}
  \end{eqnarray*}
   for $j=1,\ldots, m$; see Lemma \ref{u_lem} below for details.
   Since the Dickson invariants $\widetilde{d}_{i,m}$ can be written as polynomials in $N(x_1),\ldots,N(x_m)$,
   the $\widetilde{u}_{2-m},\ldots,\widetilde{u}_{m-2}$ are redundant, as long as we include $\xi_1,\ldots,\xi_{2m-2}$ in our generating set
   (the invariants $\widetilde{u}_{-1}$, $\widetilde{u}_0$, $\widetilde{u}_1$ require $\xi_{2m-1}$ to rewrite).
   Define $\mathcal{H}:=\{N(x_1),\ldots,N(x_m),N_m(y_m),\ldots, N_1(y_1) \}$ and let $A$ denote the algebra generated $\mathcal{H}$. 
   Note that $A$ is the ring of invariants  for the upper triangular unipotent subgroup of $\GL_{2m}(\F_q)$, each $d_{i,2m}\in A$ and $\mathcal{H}$
   is a homogeneous system of parameters. Furthermore, $\F_q[V]^{USp_{2m}}$ is a finite $A$-module and the module generators can be 
   taken to be monomials in the $\xi_i$. 
   The Bonnaf\'e-Kemper relations \cite[page 105]{BK2011} translate into relations which allow us to rewrite powers of the $\xi_i$
   in terms of elements of $A$ and powers of $\xi_j$ of lower total degree.
   In particular, the translation of  relation $R_{m-j}$  rewrites $\xi_{2j+1}^{q^{m-j-1}}$ for $j=0,\ldots,m-2$, the translation of 
   $R_{m-j}^-$ rewrites $\xi_{2j}^{q^{m-j}}$ for $j=1,\ldots, m-2$ and the $R_1^+$ rewrites $\xi_{2m-2}^q$; compare with \cite[Theorem 2.4]{BK2011}.
   Using these relations, we see that 
   $\{\xi_1^{a_1}\cdots \xi_{2m-2}^{a_{2m-2}}\mid a_i<q^{\lceil{m-i/2}\rceil}\}$ is a set of $A$-module generators
   for $\F_q[V]^{USp_{2m}}$. Observe that the product of the degrees of the elements in $\mathcal{H}$ is $q^{m(2m-1)}$, the 
   order of $USp_{2m}$ is $q^{m^2}$ and the number of module generators is $q^{m(m-1)}$. Therefore $\F_q[V]^{USp_{2m}}$
   is Cohen-Macaulay, see \cite[Theorem 3.7.1]{DK2002} or \cite[Corollary 3.1.4]{CW2011}. Hence the rewriting relations generate the ideal of relations
   and $\F_q[V]^{USp_{2m}}$ is a complete intersection.
   \end{proof}
   
The above is consistent with the calculation of $\F_q[V]^{USp_4}$ in \cite{FF2017}.
Since $USp_{2m}$ is a Sylow $p$-subgroup for all of the standard parabolic subgroups of $\Sp_{2m}(\F_q)$, the following is a consequence of
\cite{CHP1991}.

\begin{coro} If $G$ is any parabolic subgroup of $\Sp_{2m}(\F_q)$, then $\F_q[V]^G$ is Cohen-Macaulay.
\end{coro}

In the following, it is useful to take $\xi_0=0$ and $\xi_{-i}^{q^j}=-\xi_i^{q^{j-i}}$ for $j\geq i>0$.

\begin{lem} \label{u_lem} For $k\leq m$ and $i,j\geq 0$, we have
$$\sum_{s=1}^kx_s^{q^i}N_k(y_s)^{q^j}=\sum^{2m-k}_{\ell=0}\xi^{q^i}_{2m-k-i-\ell+j}\widetilde{d}^{q^j}_{\ell,2m-k}.$$
\end{lem}
\begin{proof}
We expand the right hand side giving
\begin{eqnarray*}
   \sum^{2m-k}_{\ell=0}\xi^{q^i}_{2m-k-i-\ell+j}\widetilde{d}^{q^j}_{\ell,2m-k}
   &=&\sum^{2m-k}_{\ell=0}\left(\sum_{s=1}^m\left(x_sy_s^{q^{2m-k-i-\ell+j}}   -y_sx_s^{q^{2m-k-i-\ell+j}}\right)\right)^{q^i}\widetilde{d}^{q^j}_{\ell,2m-k}\\
   &=&\sum^{2m-k}_{\ell=0}\left(\sum_{s=1}^mx_s^{q^i}y_s^{2m-k-\ell+j}-y^{q^i}_sx_s^{2m-k-\ell+j}\right)\widetilde{d}^{q^j}_{\ell,2m-k}\\
   &=&\sum_{s=1}^m(x_s^{q^i}N_k(y_s)^{q^j}-y_s^{q^i}N_k(x_s)^{q^j})\\
   &=&\sum_{s=1}^kx_s^{q^i}N_k(y_s)^{q^j},
 \end{eqnarray*}
as required.
\end{proof}

\begin{thm} \label{max_para_gen}The ring of invariants $\F_q[V]^{\maxpara_k}$ is generated by
$\xi_1,\ldots,\xi_{2m-1}$, $d_{1,2m},\ldots,d_{k,2m}$, $\widetilde{d}_{1,k},\ldots,\widetilde{d}_{k,k}$
and $\widetilde{d}_{1,2m-k},\ldots,\widetilde{d}_{m-k,2m-k}$. 
\end{thm}
\begin{proof} Using Corollary \ref{para_cor}, we have $\F_q[V]^{\maxpara_k}=(\F_q[V]^{\Sp_{2m}(\F_q)_{U_k}})^{\GL_k(\F_q)}.$
A generating set for $\F_q[V]^{\Sp_{2m}(\F_q)_{U_k}}$ was given in Theorem~\ref{stab_sub_inv}.
Let $S$ denote the subalgebra of $\F_q[V]^{\Sp_{2m}(\F_q)_{U_k}}$ generated by $\{x_1,\ldots,x_k,N_k(y_1),\ldots,N_k(y_k)\}$ and
let $W$ denote the span of $\{e_1,\ldots,e_k,f_k,\ldots,f_1\}$ so that $\F_q[W]=\F_q[y_1,\ldots,y_k,x_k,\ldots,x_1]$.
The action of $\GL_k(\F_q)$ on $W$ is a vector/covector action and the algebra homomorphism from $\F_q[W]$ to
$S$ which takes $y_i$ to $N_k(y_i)$ and fixes $x_i$ is $\GL_k(\F_q)$-equivariant. Therefore, we can construct generators for
$S^{\GL_{k}(\F_q)}$ by substituting $N_k(y_i)$ for $y_i$ in the generators of $\F_q[W]^{\GL_k(\F_q)}$ given in \cite{CW2017}.
Let $A$ denote the subalgebra of  $\F_q[V]^{\Sp_{2m}(\F_q)_{U_k}}$ generated by 
$\{\widetilde{d}_{1,2m-k},\ldots,\widetilde{d}_{2m-2k,2m-k}\}$.
Referring to Theorem \ref{stab_sub_inv}, we see that $\{x_1,\ldots,x_k,N_k(y_1),\ldots,N_k(y_k), 
\widetilde{d}_{1,2m-k},\ldots,\widetilde{d}_{2m-2k,2m-k}\}$ is a homogeneous system of parameters for
$ \F_q[V]^{\Sp_{2m}(\F_q)_{U_k}}$ and that $ \F_q[V]^{\Sp_{2m}(\F_q)_{U_k}}$  is a free finitely generated $(S\otimes A)$-module.
Furthermore, the $(S\otimes A)$-module generators can be taken to be monomials in the $\xi_i$ and $\widetilde{d}_{j,2m-k}$ for $j>2m-2k$.
Since the $\widetilde{d}_{j,2m-k}$ are $\GL_k(\F_q)$-invariant, we have
$(S\otimes A)^{\GL_k(\F_q)}=S^{\GL_k(\F_q)}\otimes A$.
 Since the $\xi_i$ are also  $\GL_k(\F_q)$-invariant, to construct a generating set for
$\F_q[V]^{\maxpara_k}=(\F_q[V]^{\Sp_{2m}(\F_q)_{U_k}})^{\GL_k(\F_q)}$, we need only adjoin
$\xi_1,\ldots,\xi_{2m-1}$ and $\widetilde{d}_{i,2m-k}$ for $i=1,2,\ldots,2m-k$ to a generating set for $S^{\GL_k(\F_q)}$.
 It was shown in \cite{CW2017} that
 $\F_q[W]^{\GL_k(\F_q)}$ is generated by Dickson invariants in the $x_i$, Dickson invariants in the $y_i$ and the Bonnaf\'e-Kemper invariants
 $u_j$ for $j=1-k,\ldots,k-1$.  Let $\bar{d}_{\ell,k}$ denote the polynomial constructed by substituting $N_k(y_i)$ for $y_i$ in the Dickson invariants
 in the $y_i$. Then $S^{\GL_k(\F_q)}$ is generated by $\{\bar{d}_{i,k},\widetilde{d}_{i,k}, \widetilde{u}_j\mid i=1,\ldots,k;\, j=1-k,\ldots,k-1\}$
 with 
$ \widetilde{u}_0=\sum_{i=1}^k x_iN_k(y_i)$,
$\widetilde{u}_{\ell}=\sum_{i=1}^k x_i^{q^{\ell}}N_k(y_i)$
 and $\widetilde{u}_{-\ell}=\sum_{i=1}^m x_iN_k(y_i)^{q^{\ell}}$ for $\ell>0$.

 For the remainder of this proof we use $H$ to denote the parabolic subgroup of $\GL_{2m}(\F_q)$ associated to the partition $(k,2m-2k,k)$
 and observe that $\maxpara_k$ is a subgroup of $H$. We claim that $\F_q[V]^H$ is generated by
 $$\{ \bar{d}_{i,k}, \widetilde{d}_{i,k}, \widetilde{d}_{j,2m-k} \mid i=1,\ldots,k;\, j=1,\ldots, 2m-2k\}.$$
 First note that we have $2m$ elements in this set and the product of the degrees equals the order of $H$.
 Thus we only need to show the set is a homogenerous system of parameters. We could do this directly but prefer to take an alternate approach.
 The set $\{\widetilde{d}_{i,k}, \widetilde{d}_{j,2m-k} \mid i=1,\ldots,k;\, j=1,\ldots, 2m-2k\}$ is the generating set for the
 parabolic subgroup of $\GL_{2m-k}(\F_q)$ associated to the partition $(2m-2k,k)$ constructed by Kuhn and Mitchell \cite[Theorem~2.2]{KM1986}.
 We can then form $H$ as the polynomial gluing of this group with $\GL_k(\F_q)$. This shows that the given set is a generating set for
 $\F_q[V]^H$ but this also means that we can replace the $\bar{d}_{i,k}$ with $d_{i,2m}$ in our generating set for $\F_q[V]^{\maxpara_k}$
 by using the Kuhn-Mitchell generators for $\F_q[V]^H$.
 
 To complete the proof we need to show that $\widetilde{u}_j$ and $\widetilde{d}_{i,2m-k}$ for $i>m-k$ are redundant.
 To show that the $\widetilde{u}_j$ are redundant we use Lemma \ref{u_lem}: taking $i=0$ and $j=0$ shows that $\widetilde{u}_0$ is redundant,
 taking $i=0$ and $j>0$ shows that $\widetilde{u}_{-j}$ is redundant, and
 taking $i>0$ and $j=0$ shows that $\widetilde{u}_{i}$ is redundant.

 Define $R:=\F_q[y_{k+1},\ldots,y_m,x_m,\ldots,x_1]$, let
 $\overline{W}_k$ denote the span of $\{x_1,\ldots,x_k\}$ and define
 $\overline{N}_k(t):=\prod_{u\in\overline{W}_k}(t-u)$.
 Let $\overline{d}_{i,2m-2k}$ denote the element of $R$ constructed by
 substituting $\overline{N}_k(y_j)$ for $y_j$ and
 $\overline{N}_k(x_j)$ for $x_j$ into the $i^{th}$ Dickson invariant in the variables
 $\{y_{k+1},\ldots,y_m,x_m,\ldots,x_{k+1}\}$.
 Define
 $$\overline{\xi}_i:=\sum_{j=k+1}^m \overline{N}_k(x_j)\overline{N}_k(y_j)^{q^i}
 -\overline{N}_k(y_j)\overline{N}_k(x_j)^{q^i}.$$
 The action of
 $\maxpara_k$ on $R$ is a polynomial gluing of $\Sp_{2m-2k}(\F_q)$ and
 $\GL_k(\F_q)$. Using the Carlisle-Kropholler generators for the symplectic invariants,
 the gluing construction and Theorem \ref{prop_thm}, we see that
 the ring $R^{\maxpara_k}$ is the complete intersection generated by
 $\widetilde{d}_{i,k}$ for $i=1,\ldots,k$, $\overline{d}_{i,2m-2k}$ for $i=1,\ldots, m-k$
   and $\overline{\xi}_i$ for $i=1,\ldots, 2m-2k-1$.
   Using the expansion of  $\overline{N}_k(t)$ in terms of the Dickson invariants and the fact that
   $\overline{N}_k(x_i)=0$ for $i\leq k$ gives
   \begin{eqnarray*} 
     \overline{\xi}_i&=&\sum_{j=1}^m\left(
     \left(\sum_{\ell=0}^k x_j^{q^{k-\ell}} \widetilde{d}_{\ell,k}\right)
     \left(\sum_{s=0}^k y_j^{q^{k-s}} \widetilde{d}_{s,k}\right)^{q^i}
    -\left(\sum_{\ell=0}^k y_j^{q^{k-\ell}} \widetilde{d}_{\ell,k}\right)
      \left(\sum_{s=0}^k x_j^{q^{k-s}} \widetilde{d}_{s,k}\right)^{q^i} \right)\\
        &=& \sum_{j=1}^m\sum_{\ell=0}^k\sum_{s=0}^k
        \left(x_j^{q^{k-\ell}}y_j^{q^{k-s+i}}-x_j^{q^{k-s+i}}y_j^{q^{k-\ell}} \right)
          \widetilde{d}_{\ell,k}\widetilde{d}_{s,k}^{q^i}
           \end{eqnarray*}
which gives
 \begin{eqnarray} \label{xib_eqn}
     \overline{\xi}_i     &=&\sum_{\ell>s-i}\xi_{\ell-s+i}^{q^{k-\ell}}
          \widetilde{d}_{\ell,k}\widetilde{d}_{s,k}^{q^i}
       + \sum_{\ell<s-i}\xi_{s-\ell-i}^{q^{k-s+i}}
         \widetilde{d}_{\ell,k}\widetilde{d}_{s,k}^{q^i}.
         \end{eqnarray}
           
   Therefore the $\overline{\xi_i}$ lie in the algebra generated by the
   $\xi_i$ and the $\widetilde{d}_{i,k}$.
   Since $\widetilde{d}_{i,2m-k}\in R^{\maxpara_k}$, these invariants can be re-written as polynomials in
   $\widetilde{d}_{i,k}$ for $i=1,\ldots,k$, $\overline{d}_{i,2m-2k}$ for $i=1,\ldots, m-k$     and $\xi_i$ for $i=1,\ldots, 2m-2k-1$.
     Using the Kuhn-Mitchel generators for the parabolic subgroup of $\GL_{2m-k}(\F_q)$
     associated to the partition $(2m-2k,k)$ and comparing degrees,
     $\overline{d}_{i,2m-2k}$ for $i=1,\ldots, m-k$ can be be rewritten using
       $\widetilde{d}_{i,2m-k}$ for $i=1,\ldots,m-k$ and $\widetilde{d}_{i,k}$
       for $i=1,\ldots,k$.
\end{proof}

\begin{coro}
The ring $\F_q[V]^{\maxpara_1}$ is a complete intersection. 
\end{coro}
\begin{proof} By Theorem \ref{max_para_gen}, $\F_q[V]^{\maxpara_1}$ is generated by $\xi_1,\ldots,\xi_{2m-1}$,
$\widetilde{d}_{1,2m-1},\ldots,\widetilde{d}_{m-1,2m-1}$, $d_{1,2m}$ and $\widetilde{d}_{1,1}=x_1^{q-1}$.
We can choose a homogeneous system of parameters 
$$\mathcal{H}=\{x_1^{q-1}, \xi_1,\ldots,\xi_{m-1},\widetilde{d}_{1,2m-1},\ldots,\widetilde{d}_{m-1,2m-1}, d_{1,2m}\}.$$
The resulting module generators are then monomials in $\xi_{m},\ldots,\xi_{2m-1}$. Since the order of the group is 
$q^{m^2}(q-1)\prod_{i=1}^{m-1}(q^{2i}-1)$, the product of the degrees of the elements of $\mathcal{H}$ is
$$q^{(3m-1)m/2}(q-1)^2\prod_{i=1}^{m-1}(q^{2i}-1)$$ and the ring is
Cohen-Macaulay we need $(q-1)q^{m(m-1)/2}$ module generators.
To prove that $\F_q[V]^{\maxpara_1}$ is a complete intersection, it is sufficient to construct relations to rewrite
$\xi_{2m-1}^{q-1}$ and $\xi_{2m-j}^{q^{j-1}}$ for $j=2,\ldots, m$.
Let $\mathfrak{h}$ denote the ideal in $\F_q[V]^{\maxpara_1}$ generated by $\mathcal{H}$.
We will complete the proof by showing that $\xi_{2m-1}^{q-1}\in\mathfrak{h}$
and $\xi_{2m-j}^{q^{j-1}}\in\mathfrak{h}$ for $j=2,\ldots, m$.

Lemma \ref{u_lem} with $k=1$, $i=0$ and $j=0$
gives
$$x_1N_1(y_1)=\sum_{\ell=0}^{2m-1}\xi_{2m-1-\ell}\widetilde{d}_{\ell,2m-1}
=\xi_{2m-1}+\sum_{\ell=1}^{m-1}\xi_{2m-1-\ell}\widetilde{d}_{\ell,2m-1}+\sum_{s=1}^{m}\xi_s\widetilde{d}_{2m-s,2m-1}.$$
Therefore $x_1N_1(y_1)-\xi_{2m-1}\in\mathfrak{h}$.
Furthermore $d_{1,2m}=\widetilde{d}^q_{1,2m-1}-N_1(y_1)^{q-1}$.
Hence $\xi_{2m-1}^{q-1}\equiv_{\mathfrak{h}}(x_1N_1)^{q-1}\in\mathfrak{h}$.

Equation \ref{xib_eqn} above with $k=1$ and $i\in\{m,\ldots,2m-2\}$
  gives $\overline{\xi}_i=\xi_i^q-\xi_{i+1}x_1^{q-1}-\xi_{i-1}x_1^{(q-1)q^i}+\xi_i x_1^{(q-1)(q^i+1)}$.
  Therefore $\xi_i^q\equiv_{\mathfrak{h}}\overline{\xi}_i$.
  Using the notation of the proof of Theorem \ref{max_para_gen} with $k=1$,
  $R^{\maxpara_1}$ is a complete intersection with the relations coming through the gluing isomorphism
  from the Carlisle-Kropholler relations for $\Sp_{2m-2k}(\F_q)$.
  Define
  $$\overline{\mathcal{H}}:=
  \{x_1^{q-1}, \overline{\xi_1},\ldots,\overline{\xi}_{m-2},\overline{d}_{1,2m-2},\ldots,\overline{d}_{m-2,2m-2}\}.$$
  an let $\overline{\mathfrak{h}}$ denote the ideal in $R^{\maxpara_1}$ generated by
  $\overline{H}$. Then $\overline{\mathfrak{h}}\subseteq \mathfrak{h}\cap R^{\maxpara_1}$.
  For $i=m,\ldots, 2m-2$, we can use the Carlisle-Kropholler relations to show
  that $\overline{\xi}_i^{q^{2m-2-i}}\in\overline{\mathfrak{h}}\subset\mathfrak{h}$.
  Therefore  $\xi_{2m-j}^{q^{j-1}}\equiv_{\mathfrak{h}}\overline{\xi}_{2m-j}^{q^{j-2}}\in\mathfrak{h}$
  for $j=2,\ldots, m$, as required.
  \end{proof}

\begin{exam} Consider the type $2$ parabolic for $m=2$.
By Theorem \ref{max_para_gen}, the invariants are generated by
$\xi_1, \xi_2, \xi_3,\widetilde{d}_{1,2},\widetilde{d}_{2,2},d_{1,4},d_{2,4}$ which have degrees
$q+1,q^2+1,q^3+1, q^2-q,q^2-1,q^4-q^3,q^4-q^2$.
There are two natural choices for a homogeneous system of parameters:
$$\mathcal{H}_1:=\{\xi_1, \xi_2,d_{1,4},d_{2,4}\},$$
$$\mathcal{H}_2:=\{\widetilde{d}_{1,2},\widetilde{d}_{2,2},d_{1,4},d_{2,4}\}.$$
For the first system of parameters, we need $q(q^2+1)(q+1)$ module generators and
for the second we need $q^2(q-1)^2(q+1)$ generators.
Using the second system of paprameters, the module generators can be taken to be
monomials in the $\xi_i$.
Using the Carlisle-Kropholler relation for $\F_q[V]^{\Sp_4(\F_q)}$, we have
$$\xi_3^q+d_{1,4}\xi_2^q+d_{2,4}\xi_1^q=\xi_1(\xi_1^{q^2+1}-\xi_2^{q+1}+\xi_1^q\xi_3)^{q-1}=\xi_1d_{4,4}.$$ 
This is essentially the translation of the relation $T_1^*$ from \cite{CW2017} into this context.
Let $\mathfrak{h}$ denote the ideal in $\F_q[V]^{G_2}$ generated by $\mathcal{H}_2$.
Since $d_{4,4}\in \F_q[\widetilde{d}_{1,2},\widetilde{d}_{2,2},d_{1,4},d_{2,4}]$,
the relation above shows $\xi_3^q\in\mathfrak{h}$.
The translation of the relation $T_1$ from \cite{CW2017} shows that $\xi_1^{q^2}\in\mathfrak{h}$.
The translation of the relation $T_{00}$ shows that $(\xi_2^{q+1}-\xi_1^q\xi_3)^{q-1}\in\mathfrak{h}$.
Conjecture 16 from \cite{CW2017} suggests two additional relations, $T_{1,0}$ and $T_{0,1}$.
These relations would translate to relations in degrees $q^4-q$ and $q^4-q^3+q^2-1$.
We can use the the relation $T_{00}$ to show that $\xi_2^{q^2-1}$ is congruent modulo $\mathfrak{h}$
to a linear combination of smaller monomials in the $\xi_i$ (using the grevlex order with
$\xi_3>\xi_2>\xi_1$). We conjecture that the relation in degree $q^4-q$ can be used to show that
$\xi_3\xi_2^{q^2-q-1}$ is also congruent to a linear combination of smaller monomials and that
the relation in degree $q^4-q^3+q^2-1$ shows that $\xi_1^q\xi_2^{q^2-q-1}$ is congruent
to a linear combination of smaller monomials. If true, by counting independent monomials,
these conjectures would show that $\F_q[V]^{\maxpara_2}$ for $m=2$ is not a complete intersection. 
We have confirmed these conjectures for $q=3$ and $q=5$ using Magma \cite{magma}.
\end{exam}

\begin{rem}
To an arbitrary parabolic subgroup of $\Sp_{2m}(\F_q)$ we can associate a partition of $2m$ : $\lambda=(r_1,\ldots,r_{\ell},2m-2k,r_{\ell},\ldots,r_1)$
with $r_1+\cdots+r_{\ell}=k$. Let $G_{\lambda}$ denote the parabolic subgroup of $\Sp_{2m}(\F_q)$ associated to $\lambda$ and
let $\GL_{\lambda}$ denote the parabolic subgroup of $\GL_{2m}(\F_q)$ associated to the partition $\lambda$.
The ring $\F_q[V]^{\GL_{\lambda}}$ is a polynomial algebra generated by a collection of Dickson invariants. 
Since $\F_q[V]^{G_{\lambda}}$ is Cohen-Macaulay, it is a free  $\F_q[V]^{\GL_{\lambda}}$-module.
We conjecture that the module generators can be taken to be monomials in the $\xi_i$ and that
 $\F_q[V]^{G_{\lambda}}$ is generated by $\xi_1,\ldots,\xi_{2m-1}$ and a subset of a generating set for $\F_q[V]^{\GL_{\lambda}}$.
\end{rem}


\section{Singular Finite Classical Groups}\label{class_groups}

In this section $\F$ will denote a finite field and $\beta$ an alternating, symmetric or Hermitian form on $V$.
We suggest \cite{Taylor1992} or \cite{Wan2002} for background material on finite classical groups.
Define $G_{\beta}:=\{g\in \GL(V)\mid \beta(gv,gw)=\beta(v,w)\, \forall v,w,\in V\}$.
Take $W_1$ to be the radical of $\beta$, in other words, $W_1:=\{u\in V\mid \beta(u,v)=0 \, \forall v \in V\}$.
Take $G_1:=\GL(W_1)$, $W_2:=V/W_1$ and $G_2$ to be the image of $G_{\beta}$ in $\GL(W_2)$.
It is a straight-forward observation that $G_{\beta}$ is the polynomial gluing $G_1\times_{\mathcal M} G_2$
with $\mathcal{M}=\Hom_{\F}(W_2,W_1)$. Furthermore, $\beta$ induces a non-degenerate form on $W_2$.
This construction reproduces the content of \cite[\S 4.7]{Hussain2011}.
Following the convention of Section \ref{glue_constr}, we denote $m=\dim(W_1)$ and $n=\dim(W_2)$.

If $\beta$ is an alternating form, then $n$ is even, $G_2$ is isomorphic to $\Sp_n(\F)$
and $G_{\beta}$ is isomorphic to $\GL_m(\F)\times_{\mathcal M}\Sp_n(\F)$.
By the work of Carlisle and Kropholler \cite[\S 8.3]{Ben1993}, we know that
$\F[W_2]^{\Sp_n(\F)}$ is a complete intersection and a unique factorization domain. 
Therefore, using Theorem \ref{prop_thm}, we see that $\F[V]^{G_{\beta}}$
is a complete intersection and a unique factorization domain.
Note that in this case, we can identify $G_{\beta}$ with the subgroup of $\GL_{n+m}(\F_q)$ 
which fixes $\xi:=x_1x_2^q-x_1^qx_2+\cdots +x_{n-1}x_n^q-x_{n-1}^qx_n$ with $\F=\F_q$.

If $\beta$ is Hermitian, then  $G_2$ is isomorphic to the unitary group
${\rm U}_n(\F)$ and $G_{\beta}$ is isomorphic to $\GL_m(\F)\times_{\mathcal M}{\rm U}_n(\F)$.
By the work of Chu and Jow \cite{CJ2006}, we know that
$\F[W_2]^{{\rm U}_n(\F)}$ is a complete intersection and a unique factorization domain. 
Therefore, using Theorem \ref{prop_thm}, we see that $\F[V]^{G_{\beta}}$
is a complete intersection and a unique factorization domain.
Note that in this case, we can identify $G_{\beta}$ with the subgroup of $\GL_{n+m}(\F_{q^2})$ 
which fixes $\xi:=x_1^{q+1}+\cdots +x_n^{q+1}$ with $\F=\F_{q^2}$.

If $\beta$ is symmetric and the characteristic of $\F$ is not $2$, then $\beta$ is the polarization of the quadratic form
$Q(v):=\beta(v,v)/2$ and $G_2$  is isomorphic to the orthogonal group ${\rm O}_n(\F,Q)$. While there are some partial results 
and conjectures concerning $\F[W_2]^{{\rm O}_n(\F,Q)}$, see \cite{Chu2001} and \cite{Chu2006},
there is no published refereed  account of the general result. 
In characteristic $2$, there is a breakdown in the correspondence between quadratic forms and symmetric forms.
We do have the work of Kropholler et al. \cite{KRS2005} which computes the ring of invariants 
for the orthogonal group associated to a non-singular quadratic form on  vector space over the field $\F_2$;
in this context, the polarization of the quadratic form is a possibly degenerate alternating form.

\begin{exam} Take $q$ odd, $n=3$ and consider the quadratic form $\Delta=x_2^2-x_1x_3$.
Then $G_2$ is ${\rm O}_3(\F_q)$. The order of ${\rm O}_3(\F_q)$ is $2q(q^2-1)$
and $\F_q[W_2]^{{\rm O}_3(\F_q)}$ is a polynomial algebra with generators in degrees 
$2$, $q+1$ and $q(q-1)$,  see \cite{Cohen1990} or \cite{Smith1999}.
The generator in degree $2$ is $\Delta$ and the generator in degree $q+1$ can be constructed by applying 
a Steenrod operation to $\Delta$.
Define
$$E:=\left\{
\begin{pmatrix}
1& 2c& c^2\\
0& 1& c\\
0& 0&1
\end{pmatrix}
\mid c\in\F_q\right\}.
$$
Then $E$ is a Sylow $p$-subgroup for ${\rm O}_3(\F_q)$.
The ring of invariants $\F_q[W_2]^E$ is the hypersurface
 generated by $\Delta$ and the orbit products of the variables,
 see \cite{HS2011} and, for the case $q=p$, \cite{Dickson_Mad}.
 Note that we can identify $W_2$ with the second symmetric power representation of
 $SL_2(\F_q)$. The action of $SL_2(\F_q)$  on $W_2$ is not faithful; the image of
 $SL_2(\F_q)$ in $W_2$ is isomorphic to $SL_2(\F_q)/\langle -1\rangle$ and has index
 $4$ in ${\rm O}_3(\F_q)$  with coset representatives
 $$\left\{\begin{pmatrix}
 (-1)^{e_1}& 0 & 0\\
 0& (-1)^{e_2} &0\\
 0&0&(-1)^{e_1} \end{pmatrix} \mid e_1, \, e_2\in\{0,1 \}
\right\}. $$
The ring of invariants of $\F_q[W_2]^{SL_2(\F_q)}$ is a hypersurface with generators in
degrees $2$, $q+1$, $q(q-1)/2$ and $q(q+1)/2$, see \cite{HS2011} and, for the case $q=p$, \cite{Dickson_Mad}.
It follows from Theorem \ref{prop_thm} that the ring of invariants for
$\GL_m(\F_q)\times_{\mathcal M}{\rm O}_3(\F_q)$ is a polynomial algebra
and that the ring of invariants for
$\GL_m(\F_q)\times_{\mathcal M}E$
and $\GL_m(\F_q)\times_{\mathcal M}(SL_2(\F_q)/\langle -1\rangle)$
are both hypersurfaces. 
 \end{exam}

\begin{exam} Take $q$ odd, $n=4$ and consider the quadratic form $u=x_2x_3-x_1x_4$.
Then $G_{2}$ is ${\rm O}^+_4(\F_q)$. The order of ${\rm O}^+_4(\F_q)$ is $2q^2(q^2-1)^2$
and $\F_q[W_2]^{{\rm O}^+_4(\F_q)}$ is a hypersurface with generators in degrees 
$2$, $q+1$, $q^2+1$, $q^3-q^2$ and $q^3-q$,  see \cite{Chu2001}.
The generator in degree $2$ is $u$ and the generators in degrees $q+1$ and $q^2+1$
can be constructed by applying Steenrod operations to $u$.
Define
$$E:=\left\{
\begin{pmatrix}
1& c_1& c_2&c_1c_2\\
0& 1& 0& c_2\\
0& 0&1& c_1\\
0&0&0&1
\end{pmatrix}
\mid (c_1,c_2)\in\F^2_q\right\}. $$
Then $E$ is a Sylow $p$-subgroup for ${\rm O}^+_4(\F_q)$.
The ring of invariants $\F_q[W_2]^E$ is a complete intersection
 generated by $u$, an element in degree $q+1$ constructed
 by applying a Steenrod operation to $u$,  and the orbit products of the variables: $N_E(x_4)$ of degree $q^3$,
 $N_E(x_2)$ and $N_E(x_2)$ of degree $q$, and $x_1$  (see \cite{FF2017}).
Let $E_1$ denote the subgroup of $E$ of order $q$ corresponding to taking $c_2=0$.
It is easy to show that $\F_q[W_2]^{E_1}$ is the hypersurface generated by $u$ and the $E_1$-orbit products of the variables;
this calculation can also be interpreted as the $n=2$ case of \cite[Theorem 2.4]{BK2011}.
Let $E_2$ denote the subgroup of $E$ of order $q$ corresponding to taking $c_1=0$.
Since $E_2$ is conjugate to $E_1$, $\F_q[V]^{E_2}$ is also a hypersurface.
Note that there is vector/covector action of $\GL_2(\F_q)$ on $W_2$ with $E_1$ as the image of the Sylow $p$-subgroup.
It follows from \cite{CW2017} that $\F_q[W_2]^{\GL_2(\F_q)}$ is generated by $u$, $u_1= x_2^qx_3-x_1^qx_4$,
$u_{-1}=x_2x_3^q-x_1x_4^q$, and Dickson invariants: $d_{1,2}$ and $d_{2,2}$ in the variables $x_1$ and $x_2$, and
 $d^*_{1,2}$ and $d^*_{2,2}$ in the variables $x_3$ and $x_4$.
Furthermore, while $\F_q[W_2]^{\GL_2(\F_q)}$ is Gorenstein, it is not a complete intersection.
Let $B$ denote the upper triangular subgroup of  ${\rm O}^+_4(\F_q)$.
It follows from \cite[Theorem 2.4]{BK2011} that $\F_q[W_2]^B$ is the complete intersection generated by
$u$, $u_1$, $u_{-1}$, $x_1^{q-1}$, and $N_E(x_i)^{q-1}$ for $i=2,3,4$.
There are two equivalence classes of maximal parabolic subgroups for ${\rm O}^+_4(\F_q)$.
One is represented by $\GL_2(\F_q)\ltimes E_2$ using the vector/covector action of $\GL_2(\F_q)$
and the other is represented by $C_2\ltimes B$, where $C_2$ acts by exchanging $x_2$ and $x_3$.
The ring $\F_q[W_1]^{C_2\ltimes B}$ is the complete intersection generated by
$u$, $u_1$, $u_{-1}$, $x_1^{q-1}$, $N_E(x_4)^{q-1}$, $N_E(x_2)^{q-1}+ N_E(x_2)^{q-1}$
and $N_E(x_2)^{q-1}N_E(x_2)^{q-1}$.
The ring $\F_q[W_2]^{\GL_2(\F_q)\ltimes E_2}$ is Cohen-Macaulay and is generated by
$d_{1,2}$, $d_{2,2}$ and the five generators for  $\F_q[W_1]^{{\rm O}^+_4(\F_q)}$.
If $H$ is any of the above subgroups of ${\rm O}^+_4(\F_q)$, the results of Section \ref{tensor_sec} can
 be used to deduce properties of the ring of invariants of $\GL_m(\F_q)\times_{\mathcal M}H$.
 \end{exam}

Recall that a group $G$ acts on $V$ as a \emph{rigid reflection group} if every isotropy subgroup  acts on $V$ as a  reflection group
(see, for example, \cite{DJ2015}).

\begin{thm} If $\F[W_2]^{G_2}$ has a homogeneous geometric separating set of size $n$
then $\F[V]^{G_{\beta}}$ has a homogeneous geometric separating set of size $m+n$ and 
$G_{\beta}$ acts on $V$ as a rigid reflection group.
\end{thm}

\begin{proof}
Since $G_1\times_{\mathcal{M}}G_2$ is a split polynomial gluing and $\F[W_1]^{G_1}=\F[W_2]^{\GL(W_1)}$
is a polynomial algebra, it follows from Theorem \ref{Split_thm} that $\F[V]^{G_{\beta}}$ has a 
homogeneous geometric separating set of size $m+n$.
It then follows from \cite{DJ2015} that  $G_{\beta}$ acts on $V$ as a rigid reflection group.
\end{proof}

\begin{exam} In this example, we consider an alternating form with $n=4$. We will show that $\F_q[W_2]^{\Sp_4(\F_q)}$
has a  homogeneous geometric separating set of size $4$ and therefore, using the above theorem, $\GL_m(\F_q)\times_{\mathcal{M}}\Sp_4(\F_q)$
acts on $V$ as a rigid reflection group. Using the work of
Carlisle and Kropholler \cite[Section 8.3]{Ben1993} we know that
$\F_{q}[W_2]^{\Sp_4(\F_{q})}$ is the hypersurface generated by
$d_{1,4},d_{2,4},\xi_1,\xi_2,\xi_3$
subject to the relation
$$\xi_3^q+d_{1,4}\xi_2^q+d_{2,4}\xi_1^q=\xi_1(\xi_1^{q^2+1}-\xi_2^{q+1}+\xi_1^q\xi_3)^{q-1}.$$ 
Write $\xi_1(\xi_1^{q^2+1}-\xi_2^{q+1}+\xi_1^q\xi_3)^{q-1}=\alpha_0+\xi_1^q(\alpha_1\xi_3+\cdots +\alpha_{q-1}\xi_3^{q-1})$
for $\alpha_i\in\F_q[\xi_1,\xi_2]$ and define $f:=d_{2,4}-(\alpha_1\xi_3+\cdots +\alpha_{q-1}\xi_3^{q-1})$ so that
$\xi_3^q=\alpha_0-d_{1,4}\xi_2^q-f\xi_1^q$. We will show that $\{f,d_{1,4},\xi_1,\xi_2\}$
is a homogeneous geometric separating set for
$\F_q[W_1]^{\Sp_4(\F_q)}$.
Let $\overline{\F}_{q}$ denote the algebraic closure of $\F_{q}$ and define $\overline{W}_2:=W_2\otimes_{\F_{q}}\overline{\F}_{q}$.
Since  $\xi_{3}^{q}$ and $\xi_{3}$ separate the same points in $\overline{W}_2$,
the value of $\xi_{3}$ at any point $w\in\overline{W}_2$ is determined by the values of $\{f,d_{1,4},\xi_1,\xi_2\}$.
It then follows from the definition of $f$ that the values of $\{f,d_{1,4},\xi_1,\xi_2\}$ on $w$ determine the value of
$d_{2,4}$ on $w$. The conclusion then follows from the fact that the generating set is a geometric separating set.
\end{exam}

\section{Parabolic Gluing}\label{para_sec}

Consider a vector space $W$ over $\F$ with a flag 
$$\mathcal{F}=(F_0(W)=\{0\},F_1(W),F_2(W),\ldots,F_{\ell}(W)=W)$$ so that $F_i(W)$ is a proper subspace of $F_{i+1}(W)$.
Let $P_{\mathcal{F}}$ denote the parabolic subgroup of $\GL(W)$ consisting of invertible linear transformations which stabilise the flag.
Define $\mathcal{M}_{\mathcal{F}}$ to be the subset of $\hom_{\F}(W,W)$ consisting of linear transformations consistent with the flag $\mathcal{F}$,
i.e.,
$$\mathcal{M}_{\mathcal{F}}=\{\phi\in\hom_{\F}(W,W)\mid \phi(F_i(W))\subseteq F_i(W) \,{\rm for\ }\, i=1,\ldots,\ell\}.$$
Choose $G_1$ and $G_2$ to be subgroups of $P_{\mathcal{F}}$ and take $W_1=W_2=W$. Then $\mathcal{M}_{\mathcal{F}}$ is a left $\F G_1$/ right $\F G_2$ 
sub-bimodule of  $\hom_{\F}(W,W)$ and we may form the {\it parabolic gluing} of $G_1$ to $G_2$ through $\mathcal{M}_{\mathcal{F}}$, which we
denote by $G_1\times_{\mathcal{F}} G_2$. 

For the rest of this section, we assume $\F=\F_q$ so that $G_1\times_{\mathcal{F}} G_2$ is a finite group.
If we choose a basis for $W$ which is consistent with the flag $\mathcal{F}$ then the matrices representing
$\mathcal{M}_{\mathcal{F}}$ are block upper-triangular with diagonal blocks of size $n_i\times n_i$
where $n_i=\dim(F_i(W)/F_{i-1}(W))$.
Denote $r_k=n_1+\cdots +n_k$.
We continue to denote the basis for $W_1^*$ by $\{y_1,\ldots,y_n\}$
and the basis for $W_2^*$ by $\{x_1,\ldots,x_n\}$ and assume the bases are consistent with the flag.
Taking the dual of the surjection from $W$ to $W/F_i(W)$ gives an inclusion of $(W/F_i(W))^*$ into $W^*$.
Define $\widetilde{U}_i$ to be the image of $(W/F_{i-1}(W))^*$ in $W_2^*$.
Thus  $\widetilde{U}_1={\rm Span}_{\F_q}\{x_1,\ldots,x_n\}$, 
$\widetilde{U}_2={\rm Span}_{\F_q}\{x_{n_1+1},\ldots,x_n\}$,
$\widetilde{U}_3={\rm Span}_{\F_q}\{x_{n_1+n_2+1},\ldots,x_n\}$, and so on.
Define $N_i(t):=\prod_{u\in\widetilde{U_i}}(t+u)$. Then
$$\mathcal{H}=\left\{x_1,\ldots, x_n\right\}\cup\left(\bigcup_{i=1}^{\ell}\left\{N_i(y_j)\mid r_{i-1}<j\le r_i\right\}\right)$$
is a generating set for $\F_q[V]^{\mathcal{M}_{\mathcal{F}}}$ and 
any basis for $V^*$ consistent with the flag is a Nakajima basis for $V$ as an $\F_q\mathcal{M}_{\mathcal{F}}$-module.
Since $G_2$ stabilises the image of  $(W/F_{i-1}(W))^*$ in $W_2^*$, we see that $G_2$ acts trivially on
$$A=\bigotimes_{i=1}^{\ell}\F_q[N_i(y_j)\mid r_{i-1}<j\le r_i]$$
and $\F_q[V]^{\M}=A\otimes \F_q[W_2]$.
However, in general, the algebra isomorphism from $\F_q[W_1]$ to $A$ which takes $y_j$ to 
$N_{\mathcal{M}_{\mathcal{F}}}(y_j)$ is not
  $G_1$-equivariant. For example, if we take $n=2$ and choose the partition $(1,1)$ then 
  $A=\F_q[N_1(y_1),N_2(y_2)]$
where $N_1(y_1)$ has degree $q^2$ and $N_2(y_2)$ has degree $q$.
If $G_1$ contains an element $g$ such that $y_1\cdot g = y_1+y_2$,
  then $$N_1(y_1)\cdot g=N_1(y_1)+N_2(y_2)^q+\left(d_{1,2}(x_1,x_2)+x_2^{q(q-1)}\right)N_2(y_2).$$
 
\begin{rem} \label{equi_rem}The algebra isomorphism 
$$\F_q[W_1/F_i(W_1)]=\F_q[y_{r_i+1},\ldots,y_n]\to\F_q[N_{i+1}(y_{r_i+1})\ldots, N_{i+1}(y_n)]$$
 which takes $y_k$ to $N_{i+1}(y_k)$ is $G_1$-equivariant -- compare with Example \ref{Fqexam}.
\end{rem}

 We will say that a sequence of homogeneous polynomials $f_1,f_2,\ldots,f_n\in \F_q[W_1]$
  is {\it consistent} with the flag
  $\mathcal{F}$ if $f_j\in\F_q[y_{r_i+1},\ldots,y_n]$ whenever $r_i<j\leq r_{i+1}$.
  In this case we write $\widetilde{f_j}$ for the polynomial formed by substituting
  $N_{i+1}(y_k)$ for $y_k$ in $f_j$.

    \begin{thm} \label{poly_para_thm}
    Suppose $\F_q[W_1]^{G_1}=\F_q[f_1,\ldots,f_n]$ where the sequence of homogeneous polynomials
    $f_1,\ldots,f_n$ is consistent with $\mathcal{F}$
    and $\F_q[W_2]^{G_2}=\F_q[h_1,\ldots,h_n]$ for homogeneous polynomials $h_1,\ldots,h_n$.
    Then $\F_q[V]^{G_1\times_{\mathcal{F}}G_2}=\F_q[\widetilde{f_1},\ldots,\widetilde{f_n},h_1,\ldots,h_n]$.
  \end{thm}
  
  \begin{proof} First observe that, since $G_1$ stabilises the flag $\mathcal{F}$,
    we have $\{\widetilde{f_1},\ldots,\widetilde{f_n}\}\subset \F[V]^{G_1\times_{\mathcal{F}}G_2}$.
    Furthermore, if $r_i<j\leq r_{i+1}$, then $\deg(\widetilde{f_j})=q^{n-i}\deg(f_j)$.

For an ideal $I\subset\F_q[V]$, let $\mathcal{V}(I)$ denote the variety
  in  $\overline{V}=V\otimes\overline{\F_q}$ cut out by $I$.
    Since  $\F_q[W_1]^{G_1}=\F_q[f_1,\ldots,f_n]$, we know that $\prod_{j=1}^n\deg(f_j)=|G_1|$
    and $\mathcal{V}(\langle f_1,\ldots,f_n\rangle)=\mathcal{V}(\langle y_1,\ldots,y_n\rangle)$
    (see Corrollary 3.2.6 and Lemma 2.6.3 of \cite{CW2011}).
    Similarly, since $\F_q[W_2]^{G_2}=\F_q[h_1,\ldots,h_n]$, we have $\prod_{j=1}^n\deg(h_j)=|G_2|$
    and 
    $\mathcal{V}(\langle h_1,\ldots,h_n\rangle)=\mathcal{V}(\langle x_1,\ldots,x_n\rangle)$.
    Observe that if $w\in\mathcal{V}(\langle x_1,\ldots,x_n\rangle)$ then
    $\widetilde{f_j}(w)=(f_j(w))^{q^{n-i}}$ whenever $r_i<j<r_{i+1}$. Hence
    $$\mathcal{V}(\langle \widetilde{f_1},\ldots,\widetilde{f_n}, h_1,\ldots,h_n\rangle)=
    \mathcal{V}(\langle y_1,\ldots,y_n, x_1,\ldots,x_n\rangle)$$
    and $\{\widetilde{f_1},\ldots,\widetilde{f_n}, h_1,\ldots,h_n\}$ is a homogeneous system of parameters.
    It follows from the construction of $\widetilde{f_j}$ that
    $$\prod_{j=1}^n\deg(\widetilde{f_j})\deg(h_j)=\left(\prod_{j=1}^n\deg(f_j)\deg(h_j)\right)\prod_{i=1}^{\ell}(q^{n-i-1})^{n_i}
    =|G_1|\cdot|G_2|\cdot|\mathcal{M}_{\mathcal{F}}|=|G_1\times_{\mathcal{F}}G_2|.$$
    Therefore $\F_q[V]^{G_1\times_{\mathcal{F}}G_2}=\F_q[\widetilde{f_1},\ldots,\widetilde{f_n},h_1,\ldots,h_n]$.
    \end{proof}

  \begin{exam} \label{poly_para_ex}
    Suppose $G_1=P_{\mathcal{F}}$.  By choosing a basis for $W$ consistent with the flag
    we can think of the elements of $P_{\mathcal{F}}$ as block upper-triangular matrices.
    Let $\widetilde{U}_{\mathcal{F}}$ denote the subgroup consisting of the elements whose
    block-diagonal entries are $1_{n_i}$. Then  $\widetilde{U}_{\mathcal{F}}$ is a $p$-group and a normal subgroup of
    $P_{\mathcal{F}}$.  A generating set for $\F_q[W_1]^{P_{\mathcal{F}}}$ is given by
$$\bigcup_{i=1}^{\ell}\left\{d_{s,n_i}\left(N_{\widetilde{U}_{\mathcal{F}}}\left(y_j\right)\mid r_{i-1}<j\le r_i\right)\mid s=1,\ldots,n_i\right\}$$
      where the Dickson invariants $d_{s,n_i}$ are evaluated on the indicated $n_i$-tuple of orbit products (the resulting polynomial
      is independent of the order). To see this observe that the polynomials are invariant, homogeneous, algebraically independent
      and the product of the degrees is the order of $P_{\mathcal{F}}$. This construction is essentially the result of iterated polynomial gluing --
      compare with \cite{Hew1996}, \cite{KM1986}, and \cite{Mui1975}.
      Since $d_{s,n_i}(N_{\widetilde{U}_{\mathcal{F}}}(y_j))\in \F_q[y_{r_{i-1}+1},\ldots,y_n]$ whenever $r_{i-1}<j\leq r_i$,
      the hypotheses of Theorem~\ref{poly_para_thm} are satisfied whenever $\F_q[W_2]^{G_2}$ is a polynomial algebra.
  \end{exam}

  \begin{thm} \label{para_ff_thm}Suppose $\F_q(W_1)^{G_1}=\F_q(f_1,\ldots,f_n)$
    where $f_1,\ldots,f_n$ is a sequence of homogeneous polynomials consistent with $\mathcal{F}$.
    Then  $\F_q(V)^{G_1\times_{\mathcal{F}}G_2}=\F_q(\widetilde{f_1},\ldots,\widetilde{f_n})\otimes \F_q(W_2)^{G_2}$.
  \end{thm}

  \begin{proof} As in the proof of Theorem \ref{poly_para_thm}, first observe that
    $\{\widetilde{f_1},\ldots,\widetilde{f_n}\}\subset \F_q[V]^{G_1\times_{\mathcal{F}}G_2}$.
    Also observe that the image of $\widetilde{f_j}-f_j^{q^{n-i}}$ lies in the ideal $\langle x_1,\ldots, x_n\rangle$
    whenever $r_i<j\leq r_{i+1}$.
    Therefore $$\F_q(\widetilde{f_1},\ldots,\widetilde{f_n})\otimes \F_q(W_2)^{G_2}\subset\F_q(V)^{G_1\times_{\mathcal{F}}G_2}.$$
    Since $G_1$  and $G_2$ are subgroups of $P_{\mathcal{F}}$, we have
    $\F_q(W_1)^{P_{\mathcal{F}}}\subset\F_q(W_1)^{G_1}=\F_q(f_1,\ldots,f_n)$
      and  $\F_q(W_2)^{P_{\mathcal{F}}}\subset\F_q(W_2)^{G_2}$.
      It then follows from Example \ref{poly_para_ex} that
      $$\F_q(V)^{P_{\mathcal{F}}\times_{\mathcal{F}} P_{\mathcal{F}}}\subset
      \F_q(\widetilde{f_1},\ldots,\widetilde{f_n})\otimes \F_q(W_2)^{G_2} \subset  \F_q(V)^{G_1\times_{\mathcal{F}}G_2}\subset \F_q(V).$$
      Therefore $\F_q(\widetilde{f_1},\ldots,\widetilde{f_n})\otimes \F_q(W_2)^{G_2} \subset \F_q(V)$is a Galois extension with
      Galois group $H$ satisfying 
      $P_{\mathcal{F}}\times_{\mathcal{F}} P_{\mathcal{F}} \ge H\ge G_1\times_{\mathcal{F}}G_2$
      and $\F_q(V)^H=\F_q(\widetilde{f_1},\ldots,\widetilde{f_n})\otimes \F_q(W_2)^{G_2} $. To complete the proof, we need to show that
      $H\le G_1\times_{\mathcal{F}}G_2$. Suppose $(h_1,\varphi,h_2)\in H$ with $h_i\in P_{\mathcal{F}}$ and $\varphi\in\mathcal{M}_{\mathcal{F}}$.
      Since $\widetilde{f_j}\in\F_q(V)^H$, we have $\widetilde{f_j}\cdot (h_1,\varphi,h_2)=\widetilde{f_j}$. It then follows from Remark \ref{equi_rem}
      that $f_j\cdot h_1=f_j$. Since this true for every $j$, by Galois Theory, we have $h_1\in G_1$. 
      For any $f\in\F_q(W_2)^{G_2}\subset \F_q(V)^H$, we have
      $f=f\cdot(h_1,\varphi,h_2)=f\cdot h_2$ and, therefore, by another application of Galois Theory, $h_2\in G_2$. Since $h_1\in G_1$ and
      $h_2\in G_2$, we have $(h_1,\varphi,h_2)\in G_1\times_{\mathcal{F}}G_2$, as required.
  \end{proof}

  \begin{exam} Choose a basis consistent with the filtration $\mathcal{F}$ and let 
  $G_1$ be a subgroup of $\GL(W_1)$ represented by a subgroup of the upper triangular unipotent matrices using this basis.  
  Then $G_1$ is a $p$-group and a subgroup of $P_{\mathcal{F}}$. Using the Campbell-Chuai construction from \cite{CC2007}
  gives a generating set $f_1,f_2,\ldots, f_n$ for the field of fractions $\F_q(W_1)^{G_1}$ which is consistent with $\mathcal{F}$. 
  If $G_2$ is any subgroup of $P_{\mathcal{F}}$ then the hypotheses of Theorem \ref{para_ff_thm} are satisfied and
  $\F_q(V)^{G_1\times_{\mathcal{F}}G_2}=\F_q(\widetilde{f_1},\ldots,\widetilde{f_n})\otimes \F_q(W_2)^{G_2}$.
    \end{exam}

\section{Diagonal Gluing}\label{diag_sec}

Suppose  we have a gluing with $G=G_1=G_2$.  Embed $G$ in $G\times G$ using the diagonal map: $g\mapsto (g,g)$.
Restricting the gluing to the image of the diagonal map gives a subgroup of $G\times_{\mathcal{M}} G$ which we denote by $G_{\mathcal{M}}$ and refer to as 
the {\it diagonal gluing}  of $G$ through $\mathcal{M}$. Note that $\Hom_{\F}(W_2,W_1)$ is a left $\F G$-module with the action given by
$\varphi \mapsto g\cdot\varphi\cdot g^{-1}$ and $\mathcal{M}$ is an $\F G$ submodule. Furthermore, $G_{\mathcal{M}}$ is isomorphic to the semi-direct product $G\ltimes \mathcal{M}$. There is an action of $G_{\mathcal{M}}$ on $V=W_1\oplus W_2$ given by
$(g,\varphi)\cdot(w_1\oplus w_2)=(g\cdot w_1+\varphi(w_2))\oplus(g\cdot w_2)$.
Since $\mathcal{M}$ is a normal subgroup of $G_{\mathcal{M}}$ , we have $\F[V]^{G_{\mathcal{M}}}=(\F[V]^{\mathcal{M}})^G$.
If there is a $G$-equivariant algebra isomorphism $\psi:\F[V]\to\F[V]^{\mathcal{M}}$ then the gluing is {\it polynomial} and
$\psi$ induces an isomorphism from $\F[V]^G$ to $\F[V]^{G_{\mathcal{M}}}$.

\begin{exam} Take $\F=\F_q$, $G=\GL_n(\F_q)$, $W_1=W_2=\F_q^n$, and $\mathcal{M}=\Hom_{\F_q}(W_1,W_2)$.
Then we have a split polynomial gluing, see Example \ref{Fqexam}. In general, computing $\F_q[\F_q^n\oplus\F_q^n]^{GL_n(\F_q)}$
is a difficult problem. However, the field of fractions $\F_q(\F_q^n\oplus\F_q^n)^{\GL_n(\F_q)}$ is rational over $\F_q$ and a
generating set is given in Section 3 of \cite{Ste1987}. Applying the gluing isomorphism gives a generating set for 
$\F_q(\F_q^n\oplus\F_q^n)^{\GL_n(\F_q)_{\mathcal{M}}}$ proving that this field is also rational over $\F_q$.
\end{exam}

\begin{exam} Take $\F=\F_q$, $G=\GL_n(\F_q)$, $W_1=\F_q^n$, $W_2=(\F_q^n)^*$, and $\mathcal{M}=\Hom_{\F_q}(W_1,W_2)$.
Then again it follows from Example \ref{Fqexam} that we have a split polynomial gluing.
A generating set for $\F_q[\F_q^n\oplus(\F_q^n)^*]^{\GL_n(\F_q)}$ was computed in \cite{CW2017}; 
the ring is Cohen-Macaulay but not a complete intersection. Applying the gluing isomorphism gives a generating set for
$\F_q[\F_q^n\oplus(\F_q^n)^*]^{\GL_n(\F_q)_{\mathcal{M}}}$. It follows from Propositions \ref{CM-prop} and \ref{CI-prop} that this ring is also Cohen-Macaulay but not a complete intersection.
\end{exam}

\begin{exam} Take $\F=\F_q$, $G=C_p$ (the cyclic group of order $p$ with $q=p^r$), 
$W_1=V_m$, $W_2=V_n$ (indecomposable $\F_qC_p$-modules) and $\mathcal{M}=\Hom_{\F_q}(W_2,W_ 1)$.
It follows from Example \ref{Fqexam} that the gluing is split polynomial.
Generating sets for $\F_q[V_m\oplus V_n]^{C_p}$ are known for $m,n\le 4$, see \cite{Weh2013}.
In each case, applying the gluing isomorphism gives a generating set for $\F_q[V_m\oplus V_n]^{(C_p)_{\mathcal{M}}}$.
By the celebrated formula of Ellinsgrud and Skjelbred (see \cite{ES1980} or  \cite[\S 3.9.2]{DK2002})
the depth of $\F_q[V_m\oplus V_n]^{C_p}$ is $4$ as long as $m+n>3$. Since the the gluing is split, this also holds for
$\F_q[V_m\oplus V_n]^{(C_p)_{\mathcal{M}}}$ (see Theorem \ref{Split_thm}).
\end{exam}

\begin{exam} Take $\F=\F_q$, $G=C_p$ and $W_1=W_2=V_n$. Let ${\bf 1}$ denote the one dimensional submodule of
$\Hom_{\F_q}(V_n,V_n)$ given by scalar multiples of the identity function. The ring $\F_q[V_n\oplus V_n]^{\bf 1}$ is 
polynomial for $n=1$, a hypersurface for $n=2$ and not Cohen-Macaulay for $n>2$. Therefore the gluing in not polynomial for $n>1$.
However, the field of fractions $\F_q(V_n\oplus V_n)^{\bf 1}$ is rational over $\F_q$. Define $u_j:=y_1x_j-y_jx_1$ for $j=2,3,\ldots,n$
and $N:=y_1^q-y_1x_1^{q-1}$. Using the Campbell-Chui construction from \cite{CC2007}, we have
$\F_q[V_n\oplus V_n]^{\bf 1}[x_1^{-1}]=\F_q[x_1,\ldots,x_n,N,u_2,\ldots,u_n][x_1^{-1}]$. 
Let $g$ denote a generator for $C_p$ and choose bases for $W_1$ and $W_2$
so that $x_jg=x_j+x_{j-1}$, $y_jg=y_j+y_{j-1}$ for $j>1$, $x_1g=x_1$ and $y_1g=y_1$. Then
${\rm Span}_{\F_q}\{u_2,\ldots,u_n\}$ is isomorphic as an $\F_qC_p$-module to $V_{n-1}^*$ and $N\in\F_q[V_n\oplus V_n]^{\bf 1}$.
Therefore $\F_q[V_n\oplus V_n]^{(C_p)_{\bf 1}}[x_1^{-1}]$ is isomorphic to $\F_q[V_{n-1}\oplus V_1\oplus V_n]^{C_p}[x_1^{-1}]$
where the isomorphism is induced by the $C_p$-equivariant map taking 
${\rm Span}_{\F_q}\{u_2,\ldots,u_n\}\oplus \F_qN \oplus V_n^*$ to $V_{n-1}^*\oplus V_1^*\oplus V_n^*$.
\end{exam}

\section*{Acknowledgments}
The first author was supported by the Fundamental Research Funds for the Central Universities (2412017FZ001) and NNSF of China (11401087).


\end{document}